\documentclass[a4paper,10pt]{article}
\usepackage{amsfonts,amsmath,amsthm,dsfont}
\author{Krzysztof Paczka\thanks{Centre of Mathematics for Applications, University of Oslo, Norway, e-mail address: k.j.paczka@cma.uio.no.}}
\title{On the properites of Poisson random measures associated with a $G$-L\'evy process\thanks{The research leading to these results has received funding from the European Research Council under the European Community's Seventh Framework Programme (FP7/2007-2013) / ERC grant agreement no [228087].}}
\newtheorem{tw}{Theorem}
\newtheorem{defin}[tw]{Definition}
\newtheorem{ass}{Assumption}
\newtheorem{lem}[tw]{Lemma}
\newtheorem{cor}[tw]{Corollary}
\newtheorem{prop}[tw]{Proposition}
\theoremstyle{definition}
\newtheorem{rem}[tw]{Remark}

\DeclareMathOperator*{\tr}{tr}

\begin{document}
\def\r0{\mathbb{R}^d_0}
    \def\E{\mathbb{E}}
    \def\GE{\hat{\mathbb{E}}}
    \def\cadlag{c\`{a}dl\`{a}g }
        \def\cliprd{C_{b,Lip}(\mathbb{R}^d)}
    \def\cliprn{C_{b,Lip}(\mathbb{R}^n)}
    \def\cliprdn{C_{b,Lip}(\mathbb{R}^{d\times n})}
    \def\cliprm{C_{b,Lip}(\mathbb{R}^m)}
    \def\cp{C^{\infty}_{p}(\mathbb{R}^n)}
    \def\clipr{C_{b,Lip}(\mathbb{R})}
    \def\lipt{Lip(\Omega_t)}
        \def\lipT{Lip(\Omega_T)}
    \def\lip{Lip(\Omega)}
    \def\qB{\langle B\rangle}
    \def\P{\mathbb{P}}
    \def\D{\mathbb{D}}
    \def\I{\mathds{1}}
    \def\N{\mathbb{N}}
        \def\n{\mathcal{N}}
    \def\L12{\mathbb{L}^{1,2}}
    \def\R{\mathbb{R}}
    \def\ae{\mathcal{A}_{\EE}}
    \def\Z{\mathbb{Z}}
        \def\A{\mathcal{A}}
    \def\H{\mathcal{H}}
        \def\tl{\tau_{\lambda}}
     \def\G{\mathcal{G}}
          \def\a{\mathcal{A}}
    \def\C{\mathbb{C}}
        \def\L{\mathbb{L}}
    \def\Q{\mathbb{Q}}
        \def\q{\mathcal{Q}}
    \def\S{\mathcal{S}}
        \def\s{\mathbb{S}}
    \def\B{\mathcal{B}}
        \def\v{\mathcal{V}}
	\def\u{\mathcal{U}}
    \def\b{\mathbb{B}}
    \def\fil{\mathbb{F}}
    \def\F{\mathcal{F}}
    \def\EE{\mathcal{E}}
    \def\m{\mathcal{M}}
    \def\p{\mathcal{P}}
    \def\ito{It\^o }
    \def\levy{L\'evy }
    \def\itolevy{It\^o-L\'evy }
        \def\levykhintchine{L\'evy-Khintchine }
    \def\dom{Dom\,\delta}
    \def\LG{L^2_G(\Omega)}
    \def\LGp{L^p_G(\Omega)}
    \def\LGT{L^2_G(\Omega_T)}
        \def\LGpl{L^2_{G_{\epsilon}}(\mathcal{F}_T)}
    \def\lgT{L^2_G(0,T)}
     \def\mgT{M^2_G(0,T)}
         \def\hSTR{\mathcal{H}^S_G([0,T]\times\mathbb{R}^d_0)}
    \def\hgTR{\mathcal{H}^2_G([0,T]\times\mathbb{R}^d_0)}
        \def\hgTRo{\mathcal{H}^1_G([0,T]\times\mathbb{R}^d_0)}
	\def\hhR{\hat{\mathcal{H}}^2_G([0,T]\times\mathbb{R}^d_0)}
    \def\hgT{\mathcal{H}^2_G(0,T)}
    \def\mg1T{M^1_G(0,T)}
	\def\Mg1T{\mathcal{M}^1_G(0,T)}
    \def\hMT{\hat{\mathcal{M}}^2_G(0,T)}
    \def\hM1T{\hat{\mathcal{M}}^1_G(0,T)}
    \def\hMTp{\hat{\mathcal{M}}^p_G(0,T)}
    \def\mgT{{M}^1_G(0,T)}
    \def\d0{\mathbb{D}(\R^+,\R^{d})}
    \def\da{\mathbb{D}(\R^+,\R^{2d})}

\maketitle
\begin{abstract}
In this paper we study the properties of the Poisson random measure and the Poisson integral associated with a $G$-\levy process. We prove that a Poisson integral is a $G$-\levy process and give the conditions which ensure that a Poisson integral belongs to a good space of random variables. In particular, we study the relation between the quasi-continuity of an integrand and the quasi-continuity of the integral. Lastly, we apply the results to establish the pathwise decomposition of a $G$-\levy process into a generalized $G$-Brownian motion and a pure-jump $G$-\levy process and prove that both processes belong to a good space of random variables. 
	

\end{abstract}

\section{Introduction}
In the last years much effort has been made to investigate the problem of model uncertainty. The problem has roots in the real life: participants of the financial markets are interested in measuring the risk of losses connected with the financial positions which they take. It is worth to note that besides the model-implicit risks resulting from the movement of prices, there is also a risk (or rather the uncertainty) that one has misspecified the model by either assuming the wrong parameters or taking a wrong class of models. The uncertainty connected with the model misspecification is very interesting from the mathematical point of view as it is on the one hand difficult to quantify, and on the other hand, it can have serious consequences to the conclusions drawn from the misspecified model. It is also very challenging as often it leads to the family of the probability measures which cannot be dominated by a single reference measure. 

Such undominated families of models were studied first by Denis and Martini in \cite{Martini} where they proposed the framework for investigating the volatility uncertainty via quasi-sure approach. At the same time Shige Peng introduced his $G$-Brownian motion, a process on a canonical space equipped with a sublinear expectation called a $G$-expectation (see \cite{Peng_GBM}). Peng constructed the $G$-expectation using the viscosity solutions of a non-linear heat equation reflecting the unknown level of volatility. Both approaches are closely connected, as it was shown in \cite{Denis_function_spaces}, and the interest spurred by both papers resulted in many interesting papers:  \cite{Peng_skrypt}, \cite{Soner_mart_rep}, \cite{Song_Mart_decomp}, \cite{Soner_quasi_anal}, \cite{Peng_general_rep}, \cite{Peng_bsde}, \cite{Girsanov}, \cite{Nutz_random}, \cite{Nutz_conditional} and many others. 

The natural generalization of both frameworks is to consider a jump processes and the uncertainty associated with the drift, the volatility and the jump component. Such jump process was first considered in \cite{Peng_levy} in which a process called a $G$-\levy process was introduced. A $G$-\levy process is a \cadlag process defined on a sublinear expectation space which has increments stationary and independent of the past and which might be decomposed into a pure-jump part and a continuous part. A sublinear expectation associated with a $G$-\levy process may be defined by some non-linear IPDE describing all three sources of uncertainty. Ren in \cite{Ren} showed that such a sublinear expectation might be represented as supremum of ordinary expectations over a relatively compact family of probability measures (which again cannot be dominated by a single reference probability measure). Nutz and Neufeld in \cite{Nutz_levy} and we in gave a characterization of the laws used in this representation. The difference between our approach and the one by Nutz and Neufeld is that we consider a strong formulation via some \itolevy integrals, whereas Nutz \emph{et al.} characterizes the laws using much weaker conditions on characteristics of a semimartingale. More information on the $G$-\levy processes can be found in Section \ref{sec_preliminaries_d}.

In this paper we investigate the properties of the Poisson random measure and a Poisson integral for $G$-\levy processes. We investigate both the distributional properties of these objects and the regularity w.r.t. omega. Among other results, we show that a Poisson integral is itself a $G$-\levy process and its regularity w.r.t. omega (i.e. quasi-continuity) is strictly connected with the regularity of an integrand. The ultimate goal of the paper is to show that the decomposition of a $G$-\levy process into a generalized $G$-Brownian motion and a process of finite variation might be done pathwise on the same sublinear expectation space. Moreover, the may require the finite variation part to be a $G$-martingale. At last we show that we cannot require the finite variation part to be a symmetric $G$-martingale, unless we extend the sublinear expectation space. Then, however, the decomposition is only meant in the distributional sense.

The structure of the paper is as follows. In Section \ref{sec_preliminaries_d} we give an introduction to the framework and present the most important results used throughout the paper. In Section \ref{sec_compensation_d} we consider two different ways of compensating a pure-jump $G$-\levy process on an auxiliary sublinear expectation space. Section \ref{sec_Poiss_random_d} is devoted to studying the continuity of an Poisson integral as an operator and investigating its distributional properties. We also establish the characterization of the spaces of integrands in terms of their tightness, uniform integrability and regularity w.r.t. $z$. In Section \ref{sec_quasi-cont_d} we investigate the regularity of the Poisson integral w.r.t. $\omega$. We give sufficient condition for the quasi-continuity of the integral in terms of the quasi-continuity of the integrand. We also establish the conditions for the quasi-continuity of the jump times of a $G$-\levy process and use them to give the necessary conditions for the quasi-continuity of the Poisson integral. Lastly, in Section \ref{sec_applications_d} we apply the previous results to establish the decomposition of a $G$-\levy process into a generalized $G$-Brownian motion and a pure-jump $G$-\levy process, which are defined without introducing the auxiliary sublinear expectation space. We also investigate in that section the possibility of compensating a pure jump process into a $G$-\levy martingale.

\section{Preliminaries}\label{sec_preliminaries_d}
	Let $\Omega$ be a given space and $\H$  be a vector lattice of real functions defined on $\Omega$, i.e. a linear space containing $1$ such that $X\in\H$ implies $|X|\in\H$. We will treat elements of $\H$ as random variables.
	\begin{defin}\label{def_sublinear_exp_d}
		\emph{A sublinear expectation} $\E$ is a functional $\E\colon \H\to\R$ satisfying the following properties
		\begin{enumerate}
			\item \textbf{Monotonicity:} If $X,Y\in\H$ and $X\geq Y$ then $\E[ X]\geq\E [Y]$.
			\item \textbf{Constant preserving:} For all $c\in\R$ we have $\E [c]=c$.
			\item \textbf{Sub-additivity:} For all $X,Y\in\H$ we have $\E [X] - \E[Y]\leq\E [X-Y]$.
			\item \textbf{Positive homogeneity:} For all $X\in\H$  we have $\E [\lambda X]=\lambda\E [X]$, $\forall\,\lambda\geq0$.
		\end{enumerate}
		The triple $(\Omega,\H,\E)$ is called \emph{a sublinear expectation space}.
	\end{defin}
	
	We will consider a space $\H$ of random variables having the following property: if\break $X_i\in\H,\ i=1,\ldots n$ then
	\[
		\phi(X_1,\ldots,X_n)\in\H,\quad \forall\ {\phi\in\cliprn},
	\]
	where $\cliprn$ is the space of all bounded Lipschitz continuous functions on $\R^n$. We will express the notions of a distribution and an independence of the random vectors using test functions in $\cliprn$.
	\begin{defin}
		An $m$-dimensional random vector $Y=(Y_1,\ldots,Y_m)$ is said to be independent of an $n$-dimensional random vector $X=(X_1,\ldots,X_n)$ if for every $\phi\in C_{b,Lip}(\R^n\times \R^m)$
		\[
			\E[\phi(X,Y)]=\E[\E[\phi(x,Y)]_{x=X}].
		\]
		Let $X_1$ and $X_2$ be  $n$-dimensional random vectors defined on sublinear random spaces\break $(\Omega_1,\H_1,\E_1)$ and $(\Omega_2,\H_2,\E_2)$ respectively. We say that $X_1$ and $X_2$ are identically distributed and denote it by $X_1 \sim X_2$, if for each $\phi\in\cliprn$ one has
		\[
			\E_1[\phi(X_1)]=\E_2[\phi(X_2)].
		\]
	\end{defin}
	Now we give the definition of $G$-\levy process (after \cite{Peng_levy}).	
	\begin{defin}\label{def_levy_d}	
		Let $X=(X_t)_{t\geq0}$ be a $d$-dimensional \cadlag process on a sublinear expectation space $(\Omega,\H, \E)$. We say that $X$ is a \levy process if:
		\begin{enumerate}
			\item $X_0=0$,
			\item for each $t,s\geq 0$ the increment $X_{t+s}-X_{t}$ is independent of $(X_{t_1},\ldots, X_{t_n})$ for every $n\in\N$ and every partition $0\leq t_1\leq\ldots\leq t_n\leq t$,
			\item the distribution of the increment $X_{t+s}-X_t,\ t,s\geq 0$ is stationary, i.e.\ does not depend on $t$.
		\end{enumerate}
		Moreover, we say that a \levy process $X$ is a $G$-\levy process, if satisfies additionally following conditions
		\begin{enumerate}\setcounter{enumi}{3}
			\item there a $2d$-dimensional \levy process $(X^c_t,X^d_t)_{t\geq0}$ such for each $t\geq0$ $X_t=X_t^c+X_t^d$, where the equality is meant in the distributional sense,
			\item\label{condition_d} processes $X^c$ and $X^d$ satisfy the following growth conditions
			\[
					\lim_{t\downarrow0} \E[|X^c_t|^3]t^{-1}=0;\quad \E[|X^d_t|]<Ct\ \textrm{for all}\ t\geq0.
			\]
		\end{enumerate}
	\end{defin}
	\begin{rem}
		The condition \ref{condition_d} implies that $X^c$ is a $d$-dimensional generalized $G$-Brownian motion (in particular, it has continuous paths), whereas the jump part $X^d$ is of finite variation.
	\end{rem}
	Peng and Hu noticed in their paper that each $G$-\levy process $X$ might be characterized by a non-local operator $G$.
	\begin{tw}[\levykhintchine representation, Theorem 35 in \cite{Peng_levy}]\label{tw_levykhintchine_d}
		Let $X$ be a $G$-\levy process in $\R^d$. For every $f\in C^3_b(\R^d)$ such that $f(0)=0$ we put
		\[
			G[f(.)]:=\lim_{\delta\downarrow 0}\, \E[f(X_{\delta})]\delta^{-1}.
		\]
		The above limit exists. Moreover, $G$ has the following \levykhintchine representation
		\[
			G[f(.)]=\sup_{(v,p,Q)\in \u}\left\{\int_{\r0} f(z)v(dz)+\langle Df(0),q\rangle + \frac{1}{2}\tr[D^2f(0)QQ^T] \right\},
		\]
		where $\r0:=\R^d\setminus\{0\}$, $\u$ is a subset $\u\subset \m(\r0)\times \R^d\times\R^{d\times d}$ and $\m(\r0)$ is a set of all Borel measures on $(\r0,\B(\r0))$. We know additionally that $\u$ has the property
		\begin{equation}\label{eq_property_of_u_d}
			\sup_{(v,p,Q)\in \u}\left\{\int_{\r0} |z|v(dz)+|q|+ \tr[QQ^T] \right\}<\infty.		
		\end{equation}

	\end{tw}
	\begin{tw}[Theorem 36 in \cite{Peng_levy}]
			Let $X$ be a $d$-dimensional $G$-\levy process.  For each $\phi\in\cliprd$, define $u(t,x):=\E[\phi(x+X_t)]$. Then $u$ is the unique viscosity solution of the following integro-PDE
			\begin{align}\label{eq_integroPDE_d}
				0=&\partial_tu(t,x)-G[u(t,x+.)-u(t,x)]\notag\\
				=&\partial_tu(t,x)-\sup_{(v,p,Q)\in \u}\left\{\int_{\r0} [u(t,x+z)-u(t,x)]v(dz)\right.\notag\\
				&\left.+\langle Du(t,x),q\rangle + \frac{1}{2}\tr[D^2u(t,x)QQ^T] \right\}
			\end{align}
			with initial condition $u(0,x)=\phi(x)$.
	\end{tw}
	
	It turns out that the set $\u$ used to represent the non-local operator $G$ fully characterize $X$, namely having $X$ we can define $\u$ satysfying eq.\ (\ref{eq_property_of_u_d}) and vice versa.
	\begin{tw}
		Let $\u$ satisfy (\ref{eq_property_of_u_d}). Consider the canonical  space $\Omega:=\d0$ of all \cadlag functions taking values in $\R^{d}$ equipped with the Skorohod topology. Then there exists a sublinear expectation $\GE$ on $\d0$ such that the canonical process $(X_t)_{t\geq0}$ is a $G$-\levy process satisfying \levykhintchine representation with the same set $\u$.
	\end{tw}
	The proof might be found in \cite{Peng_levy} (Theorem 38 and 40). We will give however the construction of $\GE$, as it is important to understand it. 
	
		Begin with defining the sets of random variables. Put
		\begin{align*}
			\lipT:=&\{\xi\in L^0(\Omega)\colon \xi=\phi(X_{t_1},X_{t_2}-X_{t_1},\ldots,X_{t_n}-X_{t_{n-1}}),\\ &\phi\in C_{b,Lip}(\R^{d\times n}),\ 0\leq t_1<\ldots<t_n<T\},
		\end{align*}
		where $X_t(\omega)=\omega_t$ is the canonical process on the space $\d0$ and $L^0(\Omega)$ is the space of all random variables, which are measurable to the filtration generated by the canonical process. We also set
		\[
			\lip:=\bigcup_{T=1}^{\infty}\, \lipT.
		\]
		Firstly, consider the random variable $\xi=\phi(X_{t+s}-X_{t})$, $\phi\in\cliprd$. We define
		\[
			\GE[\xi]:=u(s,0),
		\]
		where $u$ is a unique viscosity solution of integro-PDE (\ref{eq_integroPDE_d}) with the initial condition $u(0,x)=\phi(x)$. For general
		\[
			\xi=\phi(X_{t_1},X_{t_2}-X_{t_1},\ldots,X_{t_n}-X_{t_{n-1}}),\quad \phi\in C_{b,Lip}(\R^{d\times n})
		\]
		we set $\GE[\xi]:=\phi_n$, where $\phi_n$ is obtained via the following iterated procedure
		\begin{align*}
			\phi_1(x_1,\ldots,x_{n-1})&=\GE[\phi(x_1,\ldots, X_{t_n}-X_{t_{n-1}})],\\
			\phi_2(x_1,\ldots,x_{n-2})&=\GE[\phi_1(x_1,\ldots, X_{t_{n-1}}-X_{t_{n-2}})],\\
			&\vdots\\
			\phi_{n-1}(x_1)&=\GE[\phi_{n-1}(x_1,X_{t_{2}}-X_{t_{1}})],\\
			\phi_n&=\GE[\phi_{n-1}(X_{t_{1}})].
		\end{align*}
		Lastly, we extend definition of $\GE[.]$ on the completion of $\lipT$ (respectively $\lip$) under the norm $\|.\|_p:=\GE[|.|^p]^{\frac{1}{p}},\ p\geq1$. We denote such a completion by $L^p_G(\Omega_T)$ (or resp. $L^p_G(\Omega)$).
		
		Note that we can equip the Skorohod space $\d0$ with the canonical filtration $\F_t:=\B(\Omega_t)$, where $\Omega_t:=\{\omega_{.\wedge t}\colon \omega\in\Omega\}$. Then using the procedure above we may in fact define the time-consistent conditional sublinear expectation $\GE[\xi|\F_t]$. Namely, w.l.o.g. we may assume that $t=t_i$ for some $i$ and then
		\[
			\GE[\xi|\F_{t_i}]:=\phi_{n-i}(X_{t_{0}},X_{t_1}-X_{t_0},\ldots,X_{t_i}-X_{t_{i-1}}).
		\]
		One can easily prove that such an operator is continuous w.r.t. the norm $\|.\|_1$ and might be extended to the whole space $L^1_G(\Omega)$. By construction above, it is clear that the conditional expectation satisfies the tower property, i.e. is dynamically consistent.
		\begin{defin}
			A stochastic process $(M_t)_{t\in[0,T]}$ is called a $G$-martingale if $M_t\in L^1_G(\Omega_t)$ for every $t\in[0,T]$ and for each $0\leq s\leq t\leq T$ one has
			\[
				M_s=\GE[M_t|\F_s].
			\]
			Moreover, a $G$-martingale $M$ is called symmetric, if $-M$ is also a $G$-martingale.
		\end{defin}

\subsection{Representation of $\GE[.]$ as an upper-expectation}
$\GE[.]$ satisfies the definition of a coherent risk measure and, as it is well known, coherent risk measures exhibit representation as a supremum of some expectations over a family of some probabilities. In \cite{Ren} it has been proved that the sublinear expectation associated with a $G$-\levy process can be represented as such an upper-expectation. Moreover, in \cite{moj} we characterized that family of probability measures as laws of some \itolevy integrals under some conditions on the family of \levy measures (see Section 3 in \cite{moj}). We will use this characterization and take the following assumption throughout this paper.
\begin{ass}\label{ass1_d}
Let a canonical process $X$ be a $G$-\levy process in $\R^d$ on a sublinear expectation space $(\d0, L^1_G(\Omega), \GE)$. Let $\u\subset \m(\r0)\times \R^d\times\R^{d\times d}$ be a set used in the \levykhintchine representation of $X$ \eqref{eq_integroPDE_d} satisfying \eqref{eq_property_of_u_d}. 

	Moreover, define the set of \levy measures in $\u$ as follows
\begin{equation}\label{eq_def_v_d}
	\v:=\{v\in  \m(\r0)\colon \exists (p,q)\in \R^d\times\R^{d\times d}\textrm{ such that }(v,p,q)\in \u\}.
\end{equation}
	We assume that there exists  $q<1$ such that
	\[
		\sup_{v\in\v}\int_{\{0<|z|< 1\}}|z|^qv(dz)<\infty.
	\]
\end{ass}

\begin{rem}\label{rem_g_v_d}
Let $\G_{\B}$ denote the set of all Borel function $g\colon\R^d\to \R^d$ such that $g(0)=0$. It might be checked that for all \levy measures $\mu\in \m(\R^d)$ which are absolutely continuous w.r.t. Lebesgue measure there exists a  exists a function $g_v\in\G_{\B}$ such that
\[
	v(B)=\mu(g_v^{-1}(B))\quad \forall B\in\B(\r0).
\] 
Moreover, we can choose the functions $g_v$ in such a way that for all $\epsilon>0$ there exists $\eta>0$ such that for all $v\in\v$ we have
$g_v^{-1}(B(0,\epsilon)^c)\subset B(0,\eta)^c$. We construct such a function explicitly in Appendix, Subsection \ref{ssec_g_v_d}.

We fix a measure $\mu$ and assume additionally that	$\int_{\r0}|z|\mu(dz)<\infty$. We may consider a different parametrizing set in the \levykhintchine formula. Namely, using
	\[
		\tilde \u:=\{(g_v,p,q)\in \G_{\B}\times \R^d\times\R^{d\times d}\colon (v,p,q)\in\u\}
	\]
	it is elementary that the  equation \eqref{eq_integroPDE_d} is equivalent to the following equation
				\begin{align}\label{eq_integroPDE2_d}
				0=&\partial_tu(t,x)-\sup_{(g,p,Q)\in \tilde\u}\left\{\int_{\r0} [u(t,x+g(z))-u(t,x)]\mu(dz)\right.\notag\\
				&\left.+\langle Du(t,x),p\rangle + \frac{1}{2}\tr[D^2u(t,x)QQ^T] \right\}.
			\end{align}
\end{rem}

Let $(\tilde \Omega,\G,\P_0)$ be a probability space carrying a Brownian motion $W$ and a \levy process with a \levy triplet $(0,0,\mu)$, which is independent of $W$. Let $N(dt,dz)$ be a Poisson random measure associated with that \levy process. Define
$N_t=\int_{\r0}zN(t,dz)$, which is finite $\P_0$-a.s. as we assume that $\mu$ integrates $|z|$. We also define the filtration generated by $W$ and $N$:
\begin{align*}
	\G_t:=&\sigma\{W_s,\ N_s\colon 0\leq s\leq t\}\vee\n;\ \n:=\{A\in\tilde \Omega\colon \P_0(A)=0\};\ \mathbb G:=(\G_t)_{t\geq0}.
\end{align*}
\begin{tw}[Theorem 11-13 and Corollary 14 in \cite{moj}]\label{tw_rep_sub_lin_d}
		Introduce a set of integrands $\a_{t,T}^{\u}$, $0 \leq t<T$, associated with $\u$ as a set of all processes $\theta=(\theta^d,\theta^{1,c}, \theta^{2,c})$ defined on $]t,T]$ satisfying the following properties:
	\begin{enumerate}
		\item $(\theta^{1,c},\theta^{2,c})$ is $\mathbb G$-adapted process and $\theta^d$ is $\mathbb G$-predictable random field on $]t,T]\times \R^d$.
		\item For $\P_0$-a.a. $\omega\in\tilde \Omega$ and a.e. $s\in]t,T]$ we have that 
		$(\theta^d(s,.)(\omega),\theta^{1,c}_s(\omega), \theta^{2,c}_s(\omega))\in \tilde\u$.
		\item $\theta$ satisfies the following integrability condition
		\[
		\E^{\P_0}\left[\int_t^T\left[|\theta^{1,c}_s|+|\theta^{2,c}_s|^2+\int_{\r0}|\theta^{d}(s,z)|\mu(dz)\right]ds\right]<\infty.
		\]
	\end{enumerate}
	For $\theta\in \a_{0,\infty}^{\u}$ denote the following \levy -\ito integral as
	\[
		B^{t,\theta}_T=\int_{t}^T\, \theta^{1,c}_sds+\int_{t}^T\, \theta^{2,c}_sdW_s+ \int_{]t,T]}\int_{\r0}\, \theta^{d}(s,z)N(ds,dz).
	\]
		Lastly, for a fixed $\phi\in \cliprd$ and fixed $T>0$ define for each $(t,x)\in[0,T]\times\R^d$
	\begin{align*}
		u(t,x)&:=\sup_{\theta\in\a_{t,T}^{\u}}\E^{\P_0} [\phi(x+B^{t,\theta}_T)].
	\end{align*}
	Then under Assumption \ref{ass1_d} $u$ is the viscosity solution of the following integro-PDE
	\begin{equation}\label{eq_IPDE_backwards_d}
			\partial_tu(t,x)+G[u(t,x+.)-u(t,x)]=0
	\end{equation}
	with the terminal condition $u(T,x)=\phi(x)$. Moreover, for every $\xi\in L^1_G(\Omega)$ we can represent the sublinear expectation in the following way	
		\[
		\GE[\xi]=\sup_{\theta\in \a^{\u}_{0,\infty}}\, \E^{\P^{\theta}}[\xi],
	\]
	where  $\P^{\theta}:=\P_0\circ (B{.}^{0,\theta})^{-1},\ \theta\in\a^{\u}_{0,\infty}$. We will introduce also the following notation $\mathfrak{P}:=\{\P^{\theta}\colon \theta \in \a^{\u}_{0,\infty}\}$.
\end{tw}
We can define the capacity $c$ associated with $\GE$.
\begin{defin}
Let $\GE[.]$ has a following representation $\GE[.]=\sup_{\P\in\mathfrak{P}}\,\E^{\P}[.]$. Then capacity $c$ associated with $\GE$ is defined as
\[
	c(A):=\sup_{\P\in\mathfrak{P}}\P(A),\quad A\in\B(\Omega).
\]
We will say that a set $A\in\B(\Omega)$ is polar if $c(A)=0$. We say that a property holds quasi-surely (q.s.) if it holds outside a polar set. 
\end{defin}
\begin{rem}\label{rem_extension_GE_d}
We can also extend our sublinear expectation to all random variables $Y$ on $\Omega_T$ (or $\Omega$) for which the following expression has sense
\[
	\GE[Y]:=\sup_{\P\in\mathfrak{P}}\, \E^{\P}[Y].
\]
We can thus can also extend the definition of the norm $\|.\|_p$ and  define following spaces
\begin{enumerate}
	\item Let $L^0(\Omega_T)$ be the space of all random variables on $\Omega_T$. Let $\L^p(\Omega_T),\ p\geq 1$ be a space of all equivalence classes of functions in $L^0(\Omega_T)$ s.t. $\|.\|_p$ norm is finite. 
		\item Let $B_b(\Omega_T)$ be the space of all  bounded random variables in $L^0(\Omega_T)$. The completion of $B_b(\Omega_T)$ in the norm $\|.\|_p$ will be denoted as $\L^p_b(\Omega_T)$.
	\item Let $C_b(\Omega_T)$ be the space of all continuous and bounded random variables in $L^0(\Omega_T)$. The completion of $C_b(\Omega_T)$ in the norm $\|.\|_p$ will be denoted as $\L^p_c(\Omega_T)$.
	\item Let $C_{b,lip}(\Omega_T)$ be the space of all Lipschitz continuous random variables in $C_b(\Omega_T)$. The completion of $C_{b,lip}(\Omega_T)$ in the norm $\|.\|_p$ will be denoted as $\L^p_{c,lip}(\Omega_T)$.
\end{enumerate}
\end{rem}

\begin{defin}
	We will say that the random variable $Y\in L^0(\Omega)$ is quasi-continuous, if for all $\epsilon>0$ there exists an open set $O$ such that $c(O)<\epsilon$ and $Y|_{O^c}$ is continuous. For convenience, we will often use the abbreviation \emph{q.c}.
\end{defin}
It is well known that the following characterization of $\L^p_b(\Omega)$ and $\L^p_c(\Omega)$ holds (see Theorem 25 in \cite{Denis_function_spaces}).
\begin{prop}\label{prop_characterization_of_Lpc+Lpg_d}
	For each $p\geq 1$ one has
\begin{align*}
	\L^p_b(\Omega)=\{Y\in \L^p(\Omega)\colon \lim_{n\to\infty}\GE[|Y|^p\I_{\{|Y|>n\}}]=0\}
\end{align*}
and
\begin{align*}
	\L^p_c(\Omega)=\{Y\in \L^p(\Omega)\colon \lim_{n\to\infty}\GE[|Y|^p\I_{\{|Y|>n\}}]=0, \ Y \textrm{ has a q.c. version}\}.
\end{align*}
\end{prop}

Ren proved in \cite{Ren} that under Assumption \ref{ass1_d} we have the following inclusion $L^p_G(\Omega_T)\subset \L^p_c(\Omega_T)$. In \cite{moj} we proved that for a $G$-\levy process with finite activity we have that $L^p_G(\Omega_T)=\L^p_c(\Omega_T)$. Modifying the same proof we can obtain this equality also in our more general framework, what is done in the following proposition and theorem.
\begin{prop}\label{prop_lip_in_Lg_d}
	We have the following inclusion
	\[C_{b,lip}(\Omega_T)\subset L^1_G(\Omega_T).\]
    As a consequence
    \[\L^p_{c,lip}(\Omega_T)\subset L^p_G(\Omega_T).\]
\end{prop}
\begin{tw}\label{L_G_equal_L_c_d}
	The space $C_{b,lip}(\Omega_T)$ is dense in $C_b(\Omega_T)$ under the norm $\GE[|.|]$. Thus $L^1_G(\Omega_T)=\L^1_c(\Omega_T)$.
\end{tw}
The proof of Proposition \ref{prop_lip_in_Lg_d} is given in Appendix as it is  similar to Propostion 18 in \cite{moj}, whereas the proof of Theorem \ref{L_G_equal_L_c_d} is skipped as it is identical to the proof of Theorem 21 in the same paper.

\section{Compensating a jump part of a $G$-\levy process}\label{sec_compensation_d}

Let $X$ be a $d$-dimensional $G$-\levy process on some sublinear expectation space $(\Omega,L^1_G(\Omega),\GE[.])$ associated with a set $\u$ via \levykhintchine decomposition. We know by definition that there exists a sublinear expectation space $(\tilde \Omega,L^1_G(\tilde \Omega),\tilde\E[.])$ and a $2d$-dimensional $G$-\levy process $(X^c,X^d)$ such that the distribution of $X^c+X^d$ is the same as $X$ and $\lim_{t\downarrow0}\tilde\E[(X^c_t)^3]t^{-1}=0$ and $\tilde \E[|X^d_t|]\leq Ct$ for some constant $C$. 

Hu and Peng showed in \cite{Peng_levy} that $\tilde \Omega$ might be chosen to be $\D_0(\R_+,\R^{2d})$. The jump part $X^d$ satisfies then the \levykhintchine formula with the set $\v\times\{0\}\times\{0\}$, where $\v$ has the definition as in eq. \eqref{eq_def_v_d}. One may ask the question how to compensate the $X^d$ to obtain a $G$-martingale. It turns out that it might be done in two different manners.

The first alternative is just to substract the expectation of $X^d_t$. It is easy to check directly by the construction of $\tilde \E$ via the viscosity solution of an IPDE that $\tilde \E[X^d_t]=t\sup_{v\in\v}\int_{\r0}z v(dz)$. Consider then the process
\[Y_t:=X_t^d-t\sup_{v\in\v}\int_{\r0}z v(dz).\]
It is easy to see that $Y$ is a $G$-\levy process associated with the \levy triple $\v\times\{0\}\times\{-\sup_{v\in\v}\int_{\r0}z v(dz)\}$ and hence it has stationary increments independent of the past and with $0$ expectation. Hence it must be a $G$-martingale. In \cite{moj2} we showed much stronger result for $G$-\itolevy integral compensated by its expectation under the assumption that the $G$-\levy process driving the integral is of finite activity.

However, it is also trivial to note that $-Y$ is a $G$-martingale iff $\v=\{v\}$, i.e. $Y$ is a classical \levy process without jump-measure uncertainty. Hence, if we want to have a compensating factor which would lead to a $G$-\levy process which is a symmetric $G$-martingale, we need to be slightly cleverer. The easiest way to do that is to introduce the drift uncertainty which would exactly compensate the jump measure uncertainty. We mainly consider a $G$-levy process $Z$ associated with the set
\[
	\{(v,0,-\int_{\r0}zv(dz))\colon v\in\v\}.
\]
We stress that usually such a process is defined on a different sublinear expectation space than $X^d$. It is not a problem as long as we are interested only in the distributional properties of a $G$-\levy process. We will return to that problem in Section \ref{sec_applications_d}.

It is easy to see that both $Z_t$ and $-Z_t$ have $0$ expectation. To see it note that $u(t,x)=\pm x$ is a viscosity solution of the folowing IPDE
\begin{align*}
	0&=\partial_t\, u(t,x)-\sup_{v\in\v}\left[\int_{\r0}[u(t,x+z)-u(t,x)]v(dz)+\langle-\int_{\r0}z\,v(dz),Du(t,x) \rangle\right],\\
	u(0,x)&=\pm x.
\end{align*}
Hence both $Z$ and $-Z$ are $G$-martingales as processes with independent stationary increments and $0$ expectation. We sumarize this section with the following Proposition.
\begin{prop}
	Let $X$ be a $G$-\levy process associated with a set $\u$. We may define a sublinear expectation $\tilde \E$ on $\D_0(\R_+,\R^{2d})$ such that the canonical process $Y_t(\omega^d,\omega^c):=(X_t^d(\omega^d),X^c_t(\omega^c)):=(\omega^d_t,\omega^c_t)$ is a $\tilde G$-\levy process (under $\tilde \E[.]$ ) associated with a set $\tilde \u$ defined as
	\[
		\tilde \u:=\{((v\otimes 0) ,(-\int_{\r0}zv(dz),p+\int_{\r0}zv(dz)), (0,q))\colon (v,p,q)\in\u\}.
	\]
	Then $X^c_t+X^d_t$ has the same distribution as $X_t$, $\lim_{t\downarrow0}\tilde\E[(X^c_t)^3]t^{-1}=0$ and $\tilde \E[|X^d_t|]\leq Ct$ for some constant $C>0$ and $X^d$ is a symmetric $\tilde G$-martingale.
\end{prop}

\section{Poisson random measure and a Poisson integral associated with a $G$-\levy process}\label{sec_Poiss_random_d}
In \cite{moj} we introduced a Poisson jump measure associated with a $G$-\levy process with finite activity. In the same paper we defined a pathwise integral $\int_{t}^T\int_{\r0}\,K(s,z)\,N(ds,dz)$ w.r.t. that measure for regularly enough random fields $K$. We proved that the integral defined in such a way has good properties: it is continuous as an operator into $L^p_G(\Omega_T),\ p=1,2$, it satisfies the \ito formula etc. The class of integrands considered in that paper was large enough to consider the martingale representation, as it was shown in \cite{moj2} (see Theorem 25). However, the analysis carried out in the two papers was done under the assumption of some continuity of the integrand w.r.t. $z$, i.e. jump size. That assumption prevented us from considering a much simpler Poisson integrals $\int_{A}\phi(z)N([0,t],dz),\ A\in \B(\r0), \ 0\notin \bar A$ and integrals w.r.t. $G$-\levy processes which might not have finite activity. We also didn't say anything about the properties of a Poisson random measure. In this section we will deal with both questions.

Some of the results in the subsequent section will be given under the following assumption:
\begin{ass}\label{ass_uniform_integ_d}
	Let the family of \levy measures have the following property: there exists $p>1$ such that
	\[
		\sup_{v\in\v}\int_{\{|z|\geq 1\}}|z|^pv(dz)<\infty.
	\]
\end{ass}

First, let us introduce a Poisson random measure for a \levy process $X$ on a canonical space $\d0$ associated with a set $\u$ via \levykhintchine formula. Let also $\v$ be a set of \levy measures considered in $\u$. For such a \levy process we introduce a Poisson random measure $L$ defined by
\[
	L(]s,t],A):=\sum_{s<u\leq t} \I_{A}(\Delta X_u),\quad 0\leq s<t,\ A\in\B(\r0),\ 0\notin \bar A.
\]
Note that $L(]s,t],A)$ is well-defined as all paths of $X$ are \cadlag functions. Moreover, it is obvious that $L(]s,t],A)\in L^0(\Omega_T)$. If we consider only interval $]0,t]$ we will often shorten the notation and write $L(t,A)$ instead of $L(]0,t],A)$.

In the same way we define the Poisson integral. Let $\phi$ be a detereministic Borel function on $\R^d$ which is finite on $A\in\B(\r0),\ 0\notin \bar A$. We introduce
\[
	\int_A \phi(z)L(t,dz):=\sum_{0<u\leq t} \phi(\Delta X_u)\I_{A}(\Delta X_u).
\]
The integral is well defined and it belongs to $L^0(\Omega_T)$. We remind that we have extended the notion of the sublinear expectation $\GE[.]$ using its representation as an upper-expectation (compare with Remark \ref{rem_extension_GE_d}). This enables us to examine the continuity w.r.t. $\GE[|.|]$-norm of the integral as an operator or to study later the distributional properties of a process $t\mapsto\int_A \phi(z)L(t,dz)$. To study those problems we will introduce and characterize the function spaces on $\r0$.

\subsection{Function spaces on $\r0$ and their characterization}
We introduce the following capacity related to $\v$.
\begin{defin}
	For a set of \levy measures $\v$ define a set function $c^{\v}$ on $\B(\r0)$ by
	\[
		c^{\v}(A):=\sup_{v\in\v}\, v(A),\quad A\in\B(\r0).
	\]
	We will call this function $\v$-capacity. We will say that a set $A$ is $\v$-polar if $c^{\v}(A)=0$. Similarly we will say that a property holds ${\v}$-quasi-surely (abbr. $\v$-q.s.) if it holds outside a $\v$-polar set. Finally, we will say that a function $f:\r0\to\R$ is $\v$-quasi-continuous (abbr. $\v$-q.c.), if for every $\epsilon>0$ there exists an open subset $O$ of $\r0$ such that $c^{\v}(O)<\epsilon$ and $f|_{O^c}$ is a continuous function.
\end{defin}
We introduce the following function spaces connected with $c^{\v}$. Let $A\in\B(\r0)$  and $p\geq1$.
\begin{enumerate}
\item $L^0(A)$ is the space of all Borel functions $f$ on $\R^d$ such that $f\equiv 0$ outside $A$.
	\item $\L^p(A,\v)$ is the space of all equivalent classes of functions  $f\in L^0(A)$ such that $\|f\|_{p,A,\v}^p:=\sup_{v\in\v}\int_A|f(z)|^pv(dz)<\infty$. Equivalent classes are taken w.r.t. the $\v$-quasi-sure equivalence.
 	\item $B_b(A)$ is space of all bounded Borel functions in $L^0(A)$. $C_b(A)$ is the space of all continuous functions in $B_b(A)$. Note that $C_b(Int\, A)=C_b(\bar A)$. 
 	\item If $0\notin \bar A$ then  $\L^p_b(A,\v)$ and $\L^p_c(A,\v)$ are defined as a completion of a $B_{b}(A)$ (respectively $C_b(A)$) under the norm $\|.\|_{p,A,\v}^p$. 
 	\item If $0\in \bar A$ then we define $\L^p_{b}(A,\v)$ and $\L^p_{b}(A,\v)$ spaces as a completion under  $\|.\|_{p,A,\v}$ norm of the following sets (respectively)
\[
	\bigcup_{\substack{B\subset A\\ 0\notin \bar B}}B_b(B)\quad\textrm{and}\quad
	\bigcup_{\substack{B\subset A\\0\notin \bar B}}C_b(B).
\]
Note that $\L^p_c(Int\, A,\v)=\L^p_c(\bar A,\v)$.
\end{enumerate}

\begin{rem}\label{rem_cv_capacity_d}
Note that $\v$-capacity is a Choquet capacity, however it is usually unnormed (or even infinite). One needs to mention that the general results from \cite{Denis_function_spaces} have been proven for normed capacities, so a priori they might not hold for our $\v$-capacity. Fortunately, many results do not require this condition. In particular, we have exactly the same characterization of spaces $\L^p_b(A,\v)$ and $\L^p_c(A,\v)$ as in Proposition \ref{prop_characterization_of_Lpc+Lpg_d}, if $0\notin \bar A$. 
\end{rem}

In order to generalize this characterization, we introduce the following natural definitions of tightness and the uniform integrability of functions w.r.t. the $c^{\v}$. 
\begin{defin}
	Fix $A\in\B(\r0)$ and $f\in \L^p(A,\v)$

	We will say that $f$ is  ${\v}$-tight if for all $\epsilon>0$ there exists a compact set $F\subset \bar A\cap \r0$ such that and $\sup_{v\in \v}\int_{F^c}|f(z)|dz<\epsilon$. 
	
	We will say that $f$ is $\v$-uniformly integrable if  $$\lim_{n\to\infty}\sup_{v\in\v}\int_{\r0} |f(z)|\I_{\{|f(z)|\geq n\}}v(dz)=0.$$
\end{defin}

\begin{prop}\label{prop_characterization_spaces_v_d}
Let $A\in \B(\r0)$.
Then we have the following characterization of spaces $\L^p_b(A,\v)$ and $\L^p_c(A,\v)$
\begin{align*}
	\L^p_b(A,\v)=\{f\in \L^p(A,\v)\colon  |f|^p \textrm{ is }{\v}\textrm{-tight and }\v\textrm{-uniformly integrable}\}
\end{align*}
and
\begin{align*}
	\L^p_c(A,\v)=\{f\in \L^p_b(A,\v)\colon f \textrm{ has a }\v\textrm{-q.c. version}\}.
\end{align*}
\end{prop}
\begin{proof}
	We will follow the ideas  Proposition 18, Proposition 24 and Theorem 25 in \cite{Denis_function_spaces}. We will proceed in three steps.

\paragraph*{Step 1.} $\L^p_b(A,\v)=\{f\in \L^p(A,\v)\colon  |f|^p \textrm{ is }{\v}\textrm{-tight and }\v\textrm{-uniformly integrable}\}$.
		
		Define $J_p:=\{f\in \L^p(A,\v)\colon  |f|^p \textrm{ is }{\v}\textrm{-tight and }\v\textrm{-uniformly integrable}\}$. Fix $f\in J_p$ and $\epsilon>0$.  By the tightness of $f$ we may define the compact and bounded away from $0$ subset $F_{\epsilon}\subset \bar A$ s.t. 
		$$\sup_{v\in\v}\int_{F_{\epsilon}^c}|f(z)|^pv(dz)<\epsilon/2.$$ 
		Define $f_{n,\epsilon}=[(f\wedge n^{1/p})\vee (-n^{1/p})]\I_{F_{\epsilon}\cap A}$. Then $f_n\in B_b(F_{\epsilon}\cap A)$ and we have by definition of $J_p$ that
		\begin{align*}
			\sup_{v\in\v}\int_{A}|f(z)-f_{n,\epsilon}(z)|^pv(dz)&\leq\sup_{v\in\v}\left\{\int_{F_{\epsilon}^c\cup A^c}|f(z)|^pv(dz)+
			\int_{F_{\epsilon}\cap A}(|f(z)|^p-n)\I_{\{|f(z)|^p>n\}}v(dz)\right\}\\
			&\leq \sup_{v\in\v}\int_{F_{\epsilon}^c}|f(z)|^pv(dz)+
			\sup_{v\in\v}\int_{\r0}|f(z)|^p\I_{\{|f(z)|^p>n\}}v(dz)<\epsilon
		\end{align*}
		for $n$ large enough.
		Hence $J_p\subset \L^p_b(A,\v)$.
		
		On the other hand, for each $f\in \L^p_b(A,\v)$ we may find a sequence $\{g_n\}_{n=1}^{\infty}$ such that $g_n\in B_b(A_n)$ $n=1,2,\ldots$, where $A_n\subset A$ satisfying $0\notin \bar A_n$ and $A_n\uparrow A$, and $\|f-g_n\|^p_{p,A,\v}\to 0$.
		
		First, we claim that each $|g_n|^p$ is $\v$-tight. The proof is straightforward. We fix $\epsilon>0$ Since $\v$ restricted to $\bar A_n$ is tight (see \cite{Peng_levy}, p. 14), we may choose a compact set $K_n$ such that $c^{\v}(K_n^c\cap \bar A_n)\leq \epsilon/(M_n)^p$, where $M_n$ is a bound of $g_n$. We put $F_n:=K_n^c\cap \bar A_n$. Then $F_n$ is compact and bounded away from $0$. We also have
		\begin{align*}
			\sup_{v\in\v}\int_{F^c_n}|g_n(z)|^pv(dz)&=\sup_{v\in\v}\int_{F^c_n\cap \bar A}|g_n(z)|^pv(dz)\leq M_n^p\sup_{v\in\v}\int_{F^c_n\cap \bar A}v(dz)<\epsilon.
		\end{align*}
		To prove that $f$ is also $\v$-tight, fix $\epsilon>0$ again and $N$ such that
			$\|f(z)-g_N(z)\|_{p,A,\v}<\epsilon^{1/p}/3$.
We also fix a compact set $F\subset \overline{ A_N}\subset\bar A$ and bounded away from $0$ set such that 
		\[
			\left(\sup_{v\in\v}\int_{F^c} |g_N(z)|^pv(dz)\right)^{1/p}<\epsilon^{1/p}/3.
		\]
		We have then the estimate
		\begin{align*}
			\left(\sup_{v\in\v}\int_{F^c} |f(z)|^pv(dz)\right)^{1/p}&=\|f(z)-f(z)\I_{F}(z)\|_{p,\r0,\v}\\&\leq\|f(z)-g_N(z)\|_{p,\r0,\v}+\|g_N(z)-f(z)\I_F(z)\|_{p,\r0,\v}\\
		&\leq\|f(z)-g_N(z)\|_{p,A,\v}	
		+\left( \sup_{v\in\v}\int_{F^c} |g_N(z)|^pv(dz)\right)^{1/p}\\&+\left(\sup_{v\in\v}\int_{F} |g_N(z)-f(z)|^pv(dz)\right)^{1/p}<\epsilon^{1/p}
		\end{align*}
		and consequently $|f|^p$ is $\v$-tight.
		
		To prove $\v$-uniform integrability of $|f|^p$, we fix $\epsilon>0$ and a compact set $F_{\epsilon}\subset \r0$ such that $\sup_{v\in\v}\int_{F_{\epsilon}^c} |f(z)|^pv(dz)<\epsilon/2$. Note that
		\begin{equation}\label{eq_characterization_lpv_d}
				\sup_{v\in\v}\int_{\r0} |f(z)|^p\I_{\{|f(z)|^p\geq n\}}v(dz)<\sup_{v\in\v}\int_{F_{\epsilon}} |f(z)|^p\I_{\{|f(z)|^p\geq n\}}v(dz)+\epsilon/2 
		\end{equation}
		 Let $y_n:=\sup_{z\in A} |g_n(z)|$ and $f_{n,\epsilon}:=[(f\wedge y_n)\vee (-y_n)]\I_{F_{\epsilon}}$. Then we have $|f-f_{n,\epsilon}|\leq |f-g_n|$ on $F_{\epsilon}$ and consequently $\|f-f_{n,\epsilon}\|^p_{p,F_{\epsilon},\v}\to 0$. Hence it is not difficult to prove that for any sequence $(\eta_n)_{n=1}^{\infty}$ tending to $\infty$ one has
		 \[
		 	\sup_{v\in\v}\int_{F_{\epsilon}}\left|f(z)-(f(z)\wedge \eta_n)\vee(-\eta_n)\right|^pv(dz)\to 0
		 \]
		 and using exactly the same arguments as in Proposition 18 in \cite{Denis_function_spaces} we get that for $n$ large enough 
		 \[
		 	\sup_{v\in\v}\int_{F_{\epsilon}} |f(z)|^p\I_{\{|f(z)|^p\geq n\}}v(dz)\leq\epsilon/2 
		 \]
		 This together with \eqref{eq_characterization_lpv_d} guarantees that $|f|^p$ is $\v$-uniformly integrable.
		 
\paragraph*{Step 2.} Each element in $\L^p_c(A,\v)$ has a q.c. version.

The argument here is exactly the same as in Proposition 24 in \cite{Denis_function_spaces}: we choose the appropriate sequence of the elements in $C_{b}(A_n)$ converging to a fixed function $f\in\L^p(A,\v)$ and define open sets of small $\v$-capacity, such that the convergence is uniform on their complement. For details, see the original paper.

\paragraph*{Step 3.} $\L^p_c(A,\v)=\{f\in \L^p_b(A,\v)\colon f \textrm{ has a }\v\textrm{-q.c. version}\}$.
First note that we may just prove the characterization for set $A$ open as $\L^p_c(Int\, A,\v)=\L^p_c(\bar A,\v)$.

By step 1 and 2 we have  $\L^p_c(A,\v)\subset\{f\in \L^p_b(A,\v)\colon f \textrm{ has a }\v\textrm{-q.c. version}\}=:J_p$.

To prove the inclusion in the other direction we fix $f\in J_p$. Define $h_n:=(f\wedge n^{1/p})\vee (-n^{1/p})$. Note that $h_n$ is $\v$-q.c. Hence we may find an open set $O_n$ with small ${\v}$-capacity s.t. $h_n$ is continuous outside $O_n$. Moreover we may assume that $O_n\subset A$, as $h^n\equiv 0$ on $A^c$ (which is a closed set). By Tietze's theorem, we extend $h_n$ to a continuous functions $g_n\in C_b(A)$ (in particular, $g_n\equiv 0$ outside $A$). Moreover, we if we choose sets $O_n$ small enough, we can get that $\|f-g_n\|_{p,A,\v}\to 0$. See the proof of Theorem 25 in \cite{Denis_function_spaces} for details. Hence,  without loss of generality we may assume that 
\begin{equation}\label{eq_characterization_Lpcv_d}
		\|f-g_n\|_{p,A,\v}\leq \frac{1}{5n}.
\end{equation}

We will use now $\v$-tightness to construct functions with supports bounded away from 0 which approximate $f$. For every $n=1,2,\ldots$  we choose a compact set $F_n\subset \bar A$ ($0\notin F_n$) such that $\left(\sup_{v\in\v}\int_{F_n^c}|f(z)|^pv(dz)\right)^{1/p}<\frac{1}{5n}$. Hence, by \eqref{eq_characterization_Lpcv_d} we get
\[
	\left(\sup_{v\in\v}\int_{F_n^c}|g_n(z)|^pv(dz)\right)^{1/p}<\frac{2}{5n}.
\]
It is now easy to construct a function $f_n$ satisfying the following conditions:
\begin{enumerate}
	\item $f_n$ is a continuous bounded function with support $A_n\subset A$ bounded away from $0$.
	\item $f_n\equiv g_n$ on $F_n$.
	\item $|f_n|\leq |g_n|$ on $F_n^c$.
\end{enumerate}

Then we have the following
\begin{align*}
	\|f_n-f\|_{p,A,\v}&\leq
	\|f_n-g_n\|_{p,A,\v}+\|g_n-f\|_{p,A,\v}\\
	&=\left(\sup_{v\in\v}\int_{F_n^c}|f_n(z)-g_n(z)|^pv(dz)\right)^{1/p}+\frac{1}{5n}\\
	&=\left(\sup_{v\in\v}\int_{F_n^c}|f_n(z)|^pv(dz)\right)^{1/p}+\left(\sup_{v\in\v}\int_{F_n^c}|g_n(z)|^pv(dz)\right)^{1/p}+\frac{1}{5n}\\
	&=2\left(\sup_{v\in\v}\int_{F_n^c}|g_n(z)|^pv(dz)\right)^{1/p}+\frac{1}{5n}<\frac{1}{n}.\qedhere
\end{align*}
\end{proof}

\subsection{Continuity of the Poisson integral as an operator}

\begin{prop}\label{prop_continity_int_d}
	For each $A\in\B(\r0),\ 0\notin \bar A$ we have that the integral $\int_A\, .\, L(t,dz)$ is a continuous operator from the space $\L^1_b(A,\v)$ to $\L^1_b(\Omega_t)$.
\end{prop}
\begin{proof}
	The proof is similar to the proofs of Theorem 27 and 28 in \cite{moj}. For the sake of completeness, we will sketch it. First we take a $\phi\in B_b(A)$. 
	
	We prove that the integral of $\phi$ is $\L^1_b(\Omega_t)$. We use the extended notion of a sublinear expectation as an upper-expectation and the characterization of $\L^1_b(\Omega_t)$ from Proposition \ref{prop_characterization_of_Lpc+Lpg_d}. 

First, we prove the continuity of the integral w.r.t. the appropriate norms.
	\begin{align}\label{eq_continuity1_d}
		\GE\left[\left|\int_A\phi(z)L(t,dz)\right|\right]&=\sup_{\theta\in\a^{\u}_{0,t}}\E^{\P_0}\left[\left|\sum_{0<u\leq t}\phi(\Delta B^{0,\theta}_u)\I_A(\Delta B^{0,\theta}_u)\right|\right]\notag\\
		&=\sup_{\theta\in\a^{\u}_{0,t}}\E^{\P_0}\left[\left|\sum_{0<u\leq t}\phi(\theta^d(u,\Delta N_u))\I_A(\theta^d(u,\Delta N_u))\right|\right]\notag\\
		&=\sup_{\theta\in\a^{\u}_{0,t}}\E^{\P_0}\left[\left|\int_0^t\int_{\r0}\phi^{\theta}(u,z)N(du,dz)\right|\right],
	\end{align}
	where $\phi^{\theta}$ is a predictable random field defined as
	\[
		\phi^{\theta}(u,z):=\phi(\theta^d(u,z))\I_A(\theta^d(u,z)).
	\]

	Define now a random measure $\pi^{\theta}_u(.)(\omega):=\mu\circ(\theta^d(u,.)(\omega))^{-1}$. By the definition of a set $\a^{\u}_{0,\infty}$ we know that for a.a. $\omega$ and $u$ we have that $\pi^{\theta}(t,\omega,.)\in\v$. Hence it is now easy to see that 
	\begin{align}\label{eq_continuity2_d}
		\GE\left[\left|\int_A\phi(z)L(t,dz)\right|\right]
		&\leq\sup_{\theta\in\a^{\u}_{0,t}}\E^{\P_0}\left[\int_0^t\int_{\r0}|\phi^{\theta}(u,z)|N(du,dz)\right]\notag\\
		&=\sup_{\theta\in\a^{\u}_{0,t}}\E^{\P_0}\left[\int_0^t\int_{\r0}|\phi^{\theta}(u,z)|\mu(dz)du\right]\notag\\
		&=\sup_{\theta\in\a^{\u}_{0,t}}\E^{\P_0}\left[\int_0^t\int_{A}|\phi(z)|\pi^{\theta}_u(dz)du\right]= t\, \sup_{v\in\v} \int_A |\phi(z)|v(dz).
	\end{align}
	The last equality is true as we may always take process $\theta^d$ which is deterministic. 
	
	Hence, if $\phi\in B_b(A)$, then $\int_{A}\phi(z)L(t,dz)\in \L^1(\Omega_t)$.
	
	Now we prove that the integral is in fact in $\L_b^1(\Omega_t)$.
	Let $K$ be a bound of $\phi$. Without the loss of generality we may assume that $K=1$. Take a set $B\in \B(\r0)$ such that $0\notin \bar B$ and $B\supset \bigcup_{v\in\v}g^{-1}_v(A)$ (it is possible by Remark \ref{rem_g_v_d}). Then for any $\theta \in \a^{\u}_{0,\infty}$ we have the following inclusion
\begin{align*}
    &\left\{\left|\sum_{s<u\leq t}\phi(\theta^d(u,\Delta N_u))\I_{A}(\theta^d(u,\Delta N_u))\right|>n\right\}\subset \left\{\sum_{s<u\leq t}\left|\phi(\theta^d(u,\Delta N_u))\I_{B}(\Delta N_u)\right|>n\right\}
    \\&\quad\subset\{u\mapsto N(u,B) \textrm{ has at least } n\textrm{ jumps in }]s,t]\}=:B_n,
\end{align*}
as the sum of jumps grows only at jump times and only by a value bounded by $1$. Introduce
\[
    C_n:=B_{n}\setminus B_{n+1}=\{u\mapsto N(u,B)  \textrm{ has } n\textrm{ jumps in the interval }]s,t]\}.
\]
Hence we have the estimate
\begin{align*}
    \GE&\left[\left|\int_A\phi(z)L(t,dz)\right|\I_{\{|\int_A\phi(z)L(t,dz)|>n\}}\right]\\
     &=\sup_{\theta\in\a^{\u}_{0,t}}\E^{\P_0}\left[\left|\sum_{u\leq t}\phi(\theta^d(u,\Delta N_u))\I_{A}(\theta^d(u,\Delta N_u))\right|\I_{\left\{\left|\sum_{u\leq t}\psi(\theta^d(u,\Delta N_u))\I_{A}(\theta^d(u,\Delta N_u))\right|>n\right\}}\right]\\
&\leq \sup_{\theta\in\a^{\u}_{0,t}}\E^{\P_0}\left[\sum_{u\leq t}\left|\phi(\theta^d(u,\Delta N_u))\I_{B}(\Delta N_u)\right|\I_{B_n}\right]\\
&\leq \sum_{m=n}^{\infty}\sup_{\theta\in\a^{\u}_{0,t}}\E^{\P_0}\left[\sum_{u\leq t}\left|\phi(\theta^d(u,\Delta N_u))\I_{B}(\Delta N_u)\right|\I_{C_n}\right]
\\
&\leq \sum_{m=n}^{\infty}\sup_{\theta\in\a^{\u}_{0,t}}\E^{\P_0}\left[m\I_{C_m}\right]
=\sum_{m=n}^{\infty}\, m\P_0(C_m), \to 0\textrm{ as }n\to\infty,
\end{align*}
because the Poisson random variable has first moment finite and the number of jumps of the Poisson process $u\mapsto N(u,B)$ in the fixed interval is Poisson-distributed. Consequently, the integral belongs to $\L^1_b(\Omega_t)$ by the characterization from Proposition \ref{prop_characterization_of_Lpc+Lpg_d}. For details see the proof of Theorem 28 in \cite{moj}.
	 
	 Now it is easy to extend this result to the whole space $\L^1_b(A,\v)$ by using \eqref{eq_continuity2_d}. 	
\end{proof}
\begin{rem}\label{rem_extenstion_of_integral_d}
	\begin{enumerate}
	\item Using Proposition \ref{prop_continity_int_d} one can prove that for each $A\in\B(\r0)$ and $\phi\in\L^1_b(A,\v)$ the following integral is well defined
	\[
		\int_{A}\phi(z)L(t,dz).
	\]
	In particular, we note that under Assumption \ref{ass_uniform_integ_d} (and of course Assumption \ref{ass1_d}), for $A=\r0$ since $(z\mapsto z) \in \L^1_b(\r0,\v)$ (compare with Proposition 28 in \cite{Denis_function_spaces}). Hence, we have that we can define the integral $\int_{\r0}z L(t,dz)$ and it is an element of $\L^1_b(\Omega)$.
	\item The extended integral is a continuous operator from $\L^1_b(\r0,\v)$ to $\L^1_b(\Omega)$.
	\item By using the same argument as in eq. \eqref{eq_continuity1_d} and \eqref{eq_continuity2_d} we may conclude that $\GE[\int_A \phi(z)L(t,dz)]=t\sup_{v\in\v}\int_A\phi(z)v(dz)$.
	\end{enumerate}
\end{rem} 

\subsection{Distributional properties of  Poisson integrals}
The second important property, which we will investigate in this section, is the distribution of a Poisson integral. Thanks to the proof of the representation of $\GE[.]$ we easily get the following Proposition.

\begin{prop}
	For each $A\in \B(\r0)$ and $\phi\in \L^1_b(A,\v)$ the stochastic process $t\mapsto\int_A \phi(z)L(t,dz)$ is a $G^{\phi,A}$-\levy process, where $G^{\phi,A}$ is a non-local operator associated with set $\u^{\phi,A}= \v^{\phi,A}\times \{0\}\times \{0\}$ via \levykhintchine formula in Theorem \ref{tw_levykhintchine_d} with
	\[
	 \v^{\phi,A}:=\{v\circ \phi^{-1}(.\cap A)\colon v\in\v\}.
	\]
\end{prop}
Before we go to the proof, note that by this proposition it is trivial that the Poisson random measure $t\mapsto L(t,A)$ is a $G$-Poisson process for any $A\in \B(\r0)$, $0\notin \bar A$.
\begin{proof}
We need to check the definition of a $G$-\levy process for a process $t\mapsto Y_t:=\int_A \phi(z)L(t,dz)$. It is obvious that $Y_0=0$. Fix $t,s\geq0$, take a partition $0\leq t_1\leq \ldots \leq t_n\leq t$ and a function $\psi \in C_{b,Lip}(\R^{d\times(n+1)})$. Then we have via representation of $\GE[.]$ that
\begin{align*}
	\GE\left[\psi(Y_{t_1},\ldots,Y_{t_n}-Y_{t_{n-1}},Y_{t+s}-Y_{t})\right]&=\sup_{\theta\in \a^{\u}_{0,t+s}}\E^{\P_0}\left[\psi(\tilde B^{0,\theta}_{t_1},\ldots,\tilde B^{t_{n-1},\theta}_{t_n},\tilde B^{t,\theta}_{t+s})\right],
\end{align*}
where 
\[
	\tilde B^{a,\theta}_{b}:= \int_a^b\int_{\r0}(\phi\cdot \I_A)\circ \theta^d(s,z)N(ds,dz).
\]
The Poisson random measure $N$ and the probability measure $\P_0$ are taken from the representation of $\GE$. It is now easy to see that $\tilde B^{a,\theta}_{b}$ is an \itolevy integral which is used to prove the representation of some sublinear expectation $\tilde\E$ associated with a $G^{\phi,A}$-\levy process. By the proof of representation of $\tilde\E$ (see Section 3 in \cite{moj}) it is easy to see that the following property holds (compare with Theorem 13 in \cite{moj}):
\begin{align*}
	\sup_{\theta\in \a^{\u}_{0,t+s}}\E^{\P_0}\left[\psi(\tilde B^{0,\theta}_{t_1},\ldots,\tilde B^{t_{n-1},\theta}_{t_n},\tilde B^{t,\theta}_{t+s})\right]&=\sup_{\theta\in \a^{\u}_{0,t}}\E^{\P_0}\left[\sup_{\theta\in \a^{\u}_{t,t+s}}\E^{\P_0}\left[\psi(x,\tilde B^{t,\theta}_{t+s})\right]|_{x=B}\right],
\end{align*}
where $	B=\left(\tilde B^{0,\theta}_{t_1},\ldots,\tilde B^{t_{n-1},\theta}_{t_n}\right)$.
Hence, again by the representation of $\GE[.]$ we get that
\begin{align*}
	\GE\left[\psi(Y_{t_1},\ldots,Y_{t_n}-Y_{t_{n-1}},Y_{t+s}-Y_{t})\right]&=\GE\left[\GE \left[\psi(x,Y_{t+s}-Y_{t})\right]|_{x=(Y_{t_1},\ldots,Y_{t_n}-Y_{t_{n-1}})}\right],
\end{align*}
therefore the increment of $Y$ is independent of the past. We prove that the increment has stationary distribution in the same way. It is now obvious that $Y$ is a $G^{\phi,A}$-\levy process.
\end{proof}

Just as in the classical case, one can consider a very simple function $\phi(z)=z\I_{A}(z)$ for some Borel set $A$ bounded away from $0$. Using a technique similar to the proof above, one can prove the following result.  
\begin{prop}
	Under Assumption \ref{ass_uniform_integ_d} for any $A\in \B(\r0)$ stochastic process $t\mapsto(X_t-\int_AzL(t,dz),\int_AzL(t,dz))$ is a $\tilde G^{A}$-\levy process, where $\tilde G^{A}$ is a non-local operator associated with set $\tilde\u^{A}$ via \levykhintchine formula in Theorem \ref{tw_levykhintchine_d} with
	\[
	 \tilde\u^{A}:=\{(v|_{A^c}\otimes v|_A, (p,0),(q,0))\colon (v,p,q)\in\u\}.
	\]
\end{prop}

The direct consequence of this proposition is the following corollary. 
\begin{cor}\label{cor_decomposition_x_d}
	Let $X$ be a $G$-\levy process defined on the canonical sublinear expectation space $(\Omega,\L^1_b(\Omega),\GE[.]),\ \Omega=\d0$ satisfying Assumption \ref{ass_uniform_integ_d}. Then the decomposition $X_t=X^c_t+X^d_t$ as in Point 4 of Definition \ref{def_levy_d} might be taken on the same sublinear expectation space with $X_t^c:=X_t-\int_{\r0}zL(t,dz)$ and $X^d_t:=\int_{\r0}zL(t,dz)$, i.e. $X^c_t$ and $X^d_t$ belong to $\L^1_b(\Omega_t)$. 
\end{cor}

\section{Quasi-continuity of Poisson integrals}\label{sec_quasi-cont_d}
In this section we  concentrate on the quasi-continuity of the Poisson integrals. In particular, we  introduce the tools which enables us to investigate when processes $X^c$ and $X^d$ used in the decomposition established in the Corollary \ref{cor_decomposition_x_d} belong to the space $L^1_G(\Omega)$. Such a particular form of the decomposition is useful, as many objects (as for example the conditional expectation) are defined only on the space $L^1_G(\Omega)$. 

First, let us deal with the simple situation of a continuous function $\phi$.
\begin{prop}\label{prop_phi_continuous_integral_too_d}
	Let $\phi\in C_{b}(A)$, $A\in\B(\r0)$ and $0\notin \bar A$. Then for all $t>0$ we have that the function $\omega^t\mapsto\int_{\r0}\phi(z)L(t,dz)(\omega^t)$ is continuous. Moreover, we also have that $\omega\mapsto\int_{\r0}\phi(z)L(t,dz)(\omega)$ has a q.c. version. Hence, $\int_{\r0}\phi(z)L(t,dz)\in L^1_G(\Omega_t)$. 
\end{prop}
The proof of this proposition is very similar to the proof of Theorem 28 in \cite{moj}, so we omit it. The remarkable thing about this proposition is that even though the integral $t\mapsto\int_{\r0}\phi(z)L(t,dz)$ is adapted, its continuity w.r.t. $\omega$ depends on what happens just after time $t$. We need to underline that the function $\omega\mapsto \int_{\r0}\phi(z)L(t,dz)(\omega)$ usually has some discontinuities even for very regular $\phi$'s. 
\begin{cor}\label{cor_int_qc_L1c_d}
	Let $\phi\in\L^1_c(\r0,\v)$. Then for every $t>0$ the integral $\int_{\r0}\phi(z)L(t,dz)\in L^1_G(\Omega_t)$.
\end{cor}
\begin{proof}
	By the definition of $\L^1_c(\r0,\v)$ there exists a sequence $(\phi_n)_n,\ \phi_n\in C_{b}(A_n)$ ($0\notin \bar A_n$) under the norm $\|.\|_{1,\r0,\v}$.  We know that the integral is a continuous operator from $\L^1_b(\r0,\v)$ to $\L^1_b(\Omega_t)$ (see Remark \ref{rem_extenstion_of_integral_d}), hence $X_n:=\int_{\r0}\phi_n(z)L(t,dz)$ converges to  $\int_{\r0}\phi(z)L(t,dz)$ in $\GE[|.|]$ norm. But we also know by Proposition \ref{prop_phi_continuous_integral_too_d} that $X_n\in L^1_G(\Omega_t)$, therefore so does its limit.
\end{proof}

Hence, we have found a sufficient condition for the Poisson integral to be quasi-continuous. However, what about the necessary conditions? In the following subsections we will establish that the regularity of the integral w.r.t. $\omega$ also implies the regularity of the integrand w.r.t. $z$. 

\subsection{The regularity of the jumps times and related quasi-continuity theorems}\label{ssec_regular_stop_times_d}
Let $A$ be an open set in $\r0$ such that $0\notin\bar A$. We will define the following stopping times
\begin{align*}
	{\tau^0_A}=\overline{\tau^0_A}:=0,\quad {\tau^k_A}:=\inf\{t>{\tau^{k-1}_A}\colon \Delta X_t \in A\},\quad \overline{\tau^k_A}:=\inf\{t>\overline{\tau^{k-1}_A}\colon \Delta X_t \in \bar A\}\quad k=1,2\ldots
\end{align*}
The stopping times are well-defined and we have that $\overline{\tau^k_A}\leq \tau^k_A$. We have also the  lemma.
\begin{lem}\label{lem_capacity_of_tau_d}
 Let $A$ be an open subset of $\r0$ such that $0\notin \bar A$ and $c^{\v}(\partial A)=0$. Then for any $t>0$ we have that
 \[
 	c(\exists\ u\in [0,t] \textrm{ s.t. }\Delta X_u\in \partial A)=c(\overline{\tau^k_A}\wedge t< \tau^k_A\wedge t)=0.
 \]
 Consequently,
  \[
 	c(\exists\ u>0 \textrm{ s.t. }\Delta X_u\in \partial A)=c(\overline{\tau^k_A}< \tau^k_A)=0.
 \]
\end{lem}
\begin{proof}
First note that
\begin{align*}
	c(\overline{\tau^k_A}\wedge t< \tau^k_A\wedge t)&=c(\exists\ u \textrm{ s.t. }{\overline{\tau^{k-1}_A}\wedge t<u\leq \tau^k_A\wedge t}\, \textrm{ and }\Delta X_u\in \partial A)\\
	&\leq c(\exists\ u\in [0,t] \textrm{ s.t. }\Delta X_u\in \partial A).
\end{align*} 
So to prove the first assertion it is sufficient to check that $c(\exists\ u\in [0,t] \textrm{ s.t. }\Delta X_u\in \partial A)=0$. Assume the opposite. Then $L(t,\partial A)\geq 1$ on a set with positive capacity. Then by the definition of $c$ there exists a measure $\Q\in\mathfrak{P}$ s.t. $L(t,\partial A)\geq 1 $ with positive $\Q$-probability. Since the Poisson random measure is non-negative, we know then that $\GE[L(t,\partial A)]\geq \E^{\Q}[L(t,\partial A)]>0$.  But by Remark \ref{rem_extenstion_of_integral_d}, Point 3, we know that $\GE[L(t,\partial A)]=t\cdot c^{\v}(\partial A)=0$. Hence \[0=c(\exists\ u\in [0,t] \textrm{ s.t. }\Delta X_u\in \partial A)=c(\overline{\tau^k_A}\wedge t< \tau^k_A\wedge t).\]

Now, note that the second assertion is a simple consequence of the first one:
\begin{align*}
	c(\overline{\tau^k_A}< \tau^k_A)&=c(\exists\ q\in\Q_{+} \textrm{ s.t. }\overline{\tau^{k-1}_A}<q< \tau^k_A) \leq\sum_{q\in\Q_{+}}\,c(\overline{\tau^{k-1}_A}<q< \tau^k_A)\\
	&\leq \sum_{q\in\Q_{+}}\,c(\overline{\tau^{k-1}_A\wedge q}< \tau^k_A\wedge q)=0,
\end{align*} 
using the first assertion of the lemma. Similarly
\[
	c(\exists\ u>0 \textrm{ s.t. }\Delta X_u\in \partial A)\leq \sum_{n=1}^{\infty} c(\exists\ u\in [0,n] \textrm{ s.t. }\Delta X_u\in \partial A)=0\qedhere
\]
\end{proof}

Using this lemma it is easy to prove the quasi-continuity of the stopping times $\tau_k^A$.
\begin{prop}\label{prop_tau_qc_d}
	 Let $A$ be an open subset of $\r0$ such that $0\notin \bar A$ and $c^{\v}(\partial A)=0$. Then for any $t>0$ the random variables $\tau_A^k$, $ \I_{\{\tau^k_A\leq t\}}$ and $\Delta X_{\tau^k_A}$ are q.c. 	 
\end{prop}
\begin{proof}
Consider a set $G=\{\tau^k_A=\overline{\tau^k_A}\}$. It is easy to see that $\omega\mapsto\tau^k_A(\omega)$ is a continuous function on $G$. On the other hand $G^c\subset \{\omega\in\Omega\colon\exists\ u>0\textrm{ s.t. }\Delta w_u\in \partial A\}=:F$. $F$ is a closed set and by Lemma \ref{lem_capacity_of_tau_d} we know that $c(F)=0$. Hence, by Lemma 3.4. in \cite{Song_hitting} for each $\epsilon>0$ there exists an open set $O^{\epsilon}$ with $c(O^{\epsilon})<\epsilon$ s.t. $F\subset O^{\epsilon}$. But $(O^{\epsilon})^c\subset G$ and therefore $\tau^k_A$ is continuous on $O^{\epsilon}$. As a consequence, $\tau^k_A$ is q.c. 

To prove that $\I_{\{\tau^k_A\leq t\}}$ and $ \Delta X_{\tau^k_A}$ are q.c. we fix $\epsilon>0$ and define the set $O^{\epsilon/2}$ as above.  Consider also a set $H:=\{\omega\in\Omega\colon \Delta \omega_t\in \bar A\}$. It is easy to check by the definition of $\P^{\theta}$ used in the definition of $c$ that this set is polar. It is also a closed set, hence there exists an open set $\tilde O^{\epsilon/2}$ containing $H$ and $c(\tilde O^{\epsilon/2})<\epsilon/2$. Now define $O:= O^{\epsilon/2}\cup \tilde O^{\epsilon/2}$ and note that $c(O)<\epsilon$.

Take any sequence $(\omega^n)_n\subset O^c$ which converges in Skorohod topology to $\omega$ (also in $O^c$ by its closedness). Then by the choice of $O^{\epsilon/2}$ we have $\tau^k_A(\omega^n)\to\tau^k_A(\omega)$.  We also know that $\tau^k_A(\omega)\neq t$ by the choice of  $\tilde O^{\epsilon/2}$, hence either $\tau^k_A(\omega^n)<t$ for large $n$ (if $ \tau^k_A(\omega)<t$) or $ \tau^k_A(\omega^n)> t$ for large $n$  (if $ \tau^k_A(\omega)>t$). In any case we have the convergence $\I_{\{\tau^k_A\leq t\}}(\omega^n)\to \I_{\{\tau^k_A\leq t\}}(\omega^n)$. Hence $\I_{\{\tau^k_A\leq t\}}$ is q.c. Similarly, by the convergence  $\omega^n\to \omega$ we get that $\Delta \omega^n_{\tau^k_A(\omega^n)}\to \Delta \omega_{\tau^k_A(\omega)}$. Hence $ \Delta X_{\tau^k_A}$ is also q.c.
\end{proof}	

We may also include another jump times in the following manner without losing the quasi-continuity.
\begin{prop}\label{prop_qc_different_stop_times_d}
		 Let $A$ be an open subset of $\r0$ such that $0\notin \bar A$ and $c^{\v}(\partial A)=0$, $0\notin \bar A$. Then for any $k>l>0$ and $C>0$ the random variable $\I_{\{\tau^l_A\leq \tau^k_A-C\}}$ is  q.c. 
\end{prop}
\begin{proof}
	Fix $\epsilon >0$. By the quasi-continuity of $\tau^k_A$ and $\tau^l_A$ we may find an open set $O_1$ s.t. $c( O_1)<\epsilon/2$ and both stopping times are continuous w.r.t. $\omega$. 	

	Let $F:=\{\tau^l_A=\tau^k_A-C\}$. We claim that $F$ is polar, too. We use again the representation theorem for sublinear expectations.
	\begin{align*}
		c(F)&=\sup_{\theta\in\a^{\u}_{0,\infty}}\P^{\theta}\left(F\right)\leq \sup_{\theta\in\a^{\u}_{0,\infty}}\P^{\theta}\left(\exists_{u>0}\ \Delta X_{u-C}\in A,\ \Delta X_u\in A\right) \\&=\sup_{\theta\in\a^{\u}_{0,\infty}}\P_0\left(\exists_{u>0}\ \Delta B^{0,\theta}_{u-C}\in A,\ \Delta B^{0,\theta}_u\in A\right).
	\end{align*}
	Let $B=\bigcup_{v\in\v} g^{-1}_v(A)$. By the construction of $g_v$ we know that $0\notin \bar B$ (see Remark \ref{rem_g_v_d} and Appendix, Subsection \ref{ssec_g_v_d}). Then
		\begin{align*}
		c(F)&\leq \P_0\left(\exists_{u>0}\ \Delta N_{u-C}\in B,\ \Delta N_u\in B\right)=0
	\end{align*}
	by standard properties of \levy processes.
	
	By the continuity of both stopping times on $O^c_1$ it is obvious that $F\cap O_1^c$ is a closed set. 
	Then, just as in Proposition \ref{prop_tau_qc_d}, we use Lemma 3.4. from \cite{Song_hitting} to find an open set $O_2$ containing $F\cap O^c_1$ with capacity $c(O_2)<\epsilon/2$. Setting $O:=O_1\cup O_2$ we can easily check via argument identical to the one at the end of the proof of Proposition \ref{prop_tau_qc_d} that $\I_{\{\tau^l_A\leq \tau^k_A-C\}}$ is continuous on $O^c$. Of course $c(O)<\epsilon$, hence the indicator is q.c.
\end{proof}

The following result will also be useful. It is also interesting by itself.
\begin{prop}\label{prop_capacity_jump_in_B_d}
	Let $A,B\in\B(\r0)$ s.t. $0\notin A$ and $C\in \B(\R_+)$. Assume that $v(A)>0$ for all $v\in\v$. Then for all $k\geq1$ we have
	\[
		c(\Delta X_{\tau^k_A}\in B)\geq \sup_{v\in\v}\frac{v(B\cap A)}{v(A)}\geq \frac{c^{\v}(B\cap A)}{c^{\v}(A)} 
	\]
	and
	\[
		c(\Delta X_{\tau^k_A}\in B,\ \tau^k_A\in C)\geq \sup_{v\in\v}\frac{v(B\cap A)\mu^{v,A,k}(C)}{v(A)}\geq \frac{c^{\v}(B\cap A)\inf_{v\in\v}\mu^{v,A,k}(C)}{c^{\v}(A)} 
	\]
	where
	$\mu^{v,A,k}$ is Erlang distribution with the shape parameter $k$ and the mean $\frac{k}{v(A)}$.
\end{prop}
\begin{proof}
	Let $\tilde \a^{\u}_{0,\infty}$ denote the subset of $ \a^{\u}_{0,\infty}$ consisting of all deterministic and constant $\theta$'s. Hence for an arbitrary $\theta\in \tilde \a^{\u}_{0,\infty}$ we have that  $\theta_t\equiv(g_{v}(.), a,\sigma)$ for some $(v,a,\sigma)\in\u$. It is easy to see that the canonical process $X$ under $\P^{\theta}$ for $\theta \in \tilde \a^{\u}_{0,\infty}$ is a classical \levy process with the \levy triplet $(a,\sigma^2,v)$. By the standard theory of \levy processes and the properties of the capacity $c$ we get then that 
		\[
		c(\Delta X_{\tau^k_A}\in B)\geq \sup_{\theta\in\tilde \a^{\u}_{0,\infty}}\P^{\theta}(\Delta X_{\tau^k_A}\in B)=\sup_{(v,a,\sigma)\in\u}\frac{v(B\cap A)}{v(A)}=\sup_{v\in\v}\frac{v(B\cap A)}{v(A)}\geq \frac{c^{\v}(B\cap A)}{c^{\v}(A)}.
	\]
	Similarly  we have that
	\begin{align*}
		c(\Delta X_{\tau^k_A}\in B,\ \tau^k_A\in C)&\geq \sup_{\theta\in\tilde \a^{\u}_{0,\infty}}\P^{\theta}(\Delta X_{\tau^k_A}\in B,\ \tau^k_A\in C)=\sup_{\theta\in\tilde \a^{\u}_{0,\infty}}\P^{\theta}(\Delta X_{\tau^k_A}\in B)\,\P^{\theta}( \tau^k_A\in C)
		\\&=\sup_{(v,a,\sigma)\in\u}\frac{v(B\cap A)\mu^{v,A,k}(C) }{v(A)}=\sup_{v\in\v}\frac{v(B\cap A)\mu^{v,A,k}(C)}{v(A)}\\
		&\geq \inf_{v\in\v}\mu^{v,A,k}(C) \frac{c^{\v}(B\cap A)}{c^{\v}(A)}.
	\end{align*}
	We have used here the standard properties of a classical \levy process such as the independence a jump magnitude from the time it occurs, the distribution of the k-th jump time of size in a set $A$ etc.
\end{proof}

Using Proposition \ref{prop_tau_qc_d} and \ref{prop_capacity_jump_in_B_d} we are able to prove the following theorem.
\begin{tw}\label{tw_qc_of_phi_d}
	Let $\phi\in L^0(\r0)$. Assume that there exists an open set $A$ containing the support of $\phi$ such that $0\notin \bar A$ and $c^{\v}(\partial A)=0$. Assume moreover that $\inf_{v\in\v}\, v(A)>0$. \begin{enumerate}
	\item If $\phi(\Delta X_{\tau^k_A}\I_{\{\tau^k_A\leq t\}})$ is. q.c. for some $t>0$ and $k\geq1$ then $\phi$ is $\v$-q.c. 
	\item If $\phi$ is $\v$-q.c.  then $\phi(\Delta X_{\tau^k_A}\I_{\{\tau^k_A\leq t\}})$ is. q.c. for all $t>0$ and $k\geq1$, 
\end{enumerate}	 
\end{tw}
\begin{proof}
	\emph{Point 1.}
	
	First note that by assumption we have that $0<\inf_{v\in \v}v(A)$ and hence $\sup_{v\in\v}\frac{k}{v(A)}<\infty$. Consequently, $\inf_{v\in\v}\mu^{v,A,k}([0,t])>0$ as $\{\mu^{v,A,k}\colon v\in\v\}$ is a family of Erlang distributions with shape parameter $k$ fixed and a bounded mean (hence the mass of distribution does not accumulate in the neighbourhood of the infinity).
	
	Fix $\epsilon>0$ and take $\eta=\epsilon\, \inf_{v\in\v}\mu^{v,A,k}([0,t])/(4  c^{\v}(A))$. Note that $c^{\v}(A)<\infty$ as $0\notin \bar A$ (compare with \cite{Peng_levy}, p. 14). By the quasi-continuity of $\Delta X_{\tau^k_A}$, $\I_{\{\tau^k_A\leq t\}}$ and $\phi(\Delta X_{\tau^k}\I_{\{\tau^k_A\leq t\}})$ there exist open sets $O_1$, $O_2$ and $O_3$ such that $c(O_i)<\eta,\ i=1,2,3$ and $\Delta X_{\tau^k_A}$, $\I_{\{\tau^k_A\leq t\}}$ and $\phi(\Delta X_{\tau^k_A}\I_{\{\tau^k_A\leq t\}})$ are continuous respectively on $O_1^c$, $O_2^c$ and $O_3^c$. %Moreover, by the proof of Proposition \ref{prop_tau_qc} we know that we can take the set $O$ in such a way that $\Delta X_t=0$ on $O^c$. 
	
	Ren proved in \cite{Ren} that the family $\mathfrak{P}$ is relatively compact hence by Prohorov's theorem there exists a compact set $K$ such that $c(K^c)<\eta$. Take then a set $F=O^c_1\cap O^c_2\cap O^c_3\cap K$. Note that $F$ is a compact set as a closed subset of a compact set $K$ and that $c(F^c)<4\eta$. By the choice of $F$ we know that both $\Delta X_{\tau^k_A}\I_{\{\tau^k_A\leq t\}}$ and $\phi(\Delta X_{\tau^k_A}\I_{\{\tau^k_A\leq t\}})$ are continuous on $F$, therefore $\phi$ is continuous on the set $J:=\{\Delta X_{\tau^k_A(\omega)}\I_{\{\tau^k_A\leq t\}}(\omega)\colon \omega \in F\}$. Then $J\subset A\cup \{0\}$ by the choice of $O^c$. $J$ is also a closed set (or even compact) as an image of a compact set $F$ under a continuous function.
	
	Note also that $\phi$ is continuous on $A^c$ as its support lies in $A$. As both $J$ and $A^c$ are closed sets, we deduce that $\phi$ is a continuous function on $J\cup A^c$, also a closed set.
	
	Our target is to show that ${\v}$-capacity of $(J\cup A^c)^c=J^c\cap A$ is small. We will do that by investigating the capacity of the following event: $\{\Delta X_{\tau^k_{A}}\in J^c,\ \tau^k_A\leq t\}$. By Proposition \ref{prop_capacity_jump_in_B_d} we have that 
	\begin{align*}
		c(\Delta X_{\tau^k_{A}}\in J^c,\ \tau^k_A\leq t)\geq  \frac{c^{\v}(J^c\cap A)}{c^{\v}(A)} \inf_{v\in\v}\mu^{v,A,k}([0,t]) 
	\end{align*}
	hence
	\[
		c^{\v}(J^c\cap A)\leq\frac{ c^{\v}(A)}{ \inf_{v\in\v}\mu^{v,A,k}([0,t]) }\,c(\Delta X_{\tau^k_{A}}\in J^c,\ \tau^k_A\leq t)
	\]
	Note also that we have the following set inclusion $F\subset\{ \Delta X_{\tau^k_{A}}\in J\}\cup\{ \tau^k_A>t \}$. Since, of course, $\{\Delta X_{\tau^k_{A}}\in J^c,\ \tau^k_A\leq t\}^c=\{ \Delta X_{\tau^k_{A}}\in J\}\cup\{ \tau^k_A> t\}$, we easily get $ \{ \Delta X_{\tau^k_{A}}\in J^c,\ \tau^k_A\leq t\}\subset F^c$ and
	\[
		c^{\v}(J^c\cap A)\leq 4\eta\, \frac{ c^{\v}(A)}{ \inf_{v\in\v}\mu^{v,A,k}([0,t]) }=\epsilon.
	\]
	Hence, we have proved the $\v$-quasi-continuity of $\phi$.
	
	\emph{Point 2.}
	
	Fix $\epsilon>0$, $k\geq 1$ and $t>0$ and choose an open subset $\tilde O_t\subset \r0$ such that $c^{\v}(\tilde O_t)<\epsilon/(2t)$ and $\phi$ is continuous on $ \tilde O_t^c$. Note that $\phi$ is also continuous on $A^c$ as the support of $\phi$ lies in $A$. Define the set
	\[
		O_t:=\{\omega\in\Omega\colon \exists\, u\leq t\ \Delta \omega_u\in \tilde O_T\cap  A\}.
	\]
	Hence $O^c_t=\{\omega\in\Omega\colon \forall\, u\leq t\ \Delta \omega_u\in \tilde O^c_t\cup A^c \}$ and it is a closed set. Note also that we can express $O_t$ in the following manner $O_t=\{L(t,\tilde O_t\cap   A)\geq 1\}$. Hence
	\begin{align*}
		\frac{\epsilon}{2}&\geq t\cdot c^{\v}(\tilde O_t\cap A)=\GE[L(t,\tilde O_t\cap A)]=\sup_{\P\in\mathfrak{P}}\E^{\P}[L(t,\tilde O_t\cap A)]\\
		&=\sup_{\P\in\mathfrak{P}}\sum_{k=1}^{\infty}\,k\,\P(\{L(t,\tilde O_t\cap A)=k\})\geq \sup_{\P\in\mathfrak{P}}\sum_{k=1}^{\infty}\,\P(\{L(t,\tilde O_t\cap A)=k\})\\
		&=\sup_{\P\in\mathfrak{P}} \P(\{L(t,\tilde O_t\cap A)\geq 1\})=c(O_t).
	\end{align*}
	Take also a closed set $F$ with $c(F^c)<\epsilon/2$ such that $\Delta X_{\tau^k_{A}}\I_{\{\tau^k_a\leq t\}}$ is continuous on $F$. Put $O:=O_t\cup F^c$. Then $\Delta X_{\tau^k_{A}}\I_{\{\tau^k_a\leq t\}}$ is continuous on $O^c$ and takes values only in $\tilde O_t^c\cup\{0\}$. Hence, $\phi(\Delta X_{\tau^k_{A}}\I_{\{\tau^k_A\leq t\}})$ is continuous on $O^c$. 
\end{proof}
\begin{rem}\label{rem_qc_phi_d}
	We might have also prove Theorem \ref{tw_qc_of_phi_d}, Point 1 assuming the quasi-continuity of $\phi(\Delta X_{\tau^k_A})$ instead of $\phi(\Delta X_{\tau^k_A}\I_{\{\tau^k_A\leq t\}})$. The proof would be a nearly identical (and we need to require only $c^{\v}(A)>0$ instead of $\inf_{v\in\v}v(A)>0$). However by Proposition \ref{prop_tau_qc_d} it is trivial that if $\phi(\Delta X_{\tau^k_A})$ is q.c. then so is $\phi(\Delta X_{\tau^k_A}\I_{\{\tau^k_A\leq t\}})$. On the other hand, one can also prove that if $\inf_{v\in\v}v(A)>0$ and $\phi\in \L_b^1(A,\v)$ then the quasi-continuity of $\phi(\Delta X_{\tau^k_{A}}\I_{\{\tau^k_A\leq t\}})$ for all $t>0$ implies the quasi-continuity of $\phi(\Delta X_{\tau^k_{A}})$. Hence Point 2 in the theorem above might have been slightly stronger.
\end{rem}

	Theorem \ref{tw_qc_of_phi_d} is very interesting, but it doesn't give us the answer to the problem, if the quasi-continuity of the Poisson integral implies the $\v$-quasi-continuity of the integrand. We have of course that
	\[
		\int_{A}\phi(z)L(t,dz)=\sum_{k=1}^{\infty}\phi(\Delta X_{\tau^k_A})\I_{\{\tau_A^k\leq t\}}
	\]
	and it is easy to see that if $\phi(\Delta X_{\tau^k_A}\I_{\{\tau_A^k\leq t\}})$ is q.c. for all $k\geq 1$ then the integral is also q.c. However, we are interested in proving the opposite relation. Even if we assumed that the integral is quasi-continuous for all $t>0$ (which seems to be a reasonable assumption), it is not trivial how to proceed with such a proof. Hence we will take slightly stronger assumptions. First, we remind the notion of a quasi-continuity of a stochastic process (as it was introduced in \cite{Song_Mart_decomp}).
	\begin{defin}
		Let $Y$ be a stochastic process. We say that $Y$ is quasi-continuous on $I$ ($I=[0,T]$ or $\R_{+}$) if for all $\epsilon>0$ there exists an open set $O$ such that $c(O)<\epsilon$ and $(t,\omega)\mapsto Y_t(\omega)$ is a continuous function on $I\times O^c$.
	\end{defin}
	\begin{rem}\label{rem_idea_of_proof_qc_d}
Assume for a moment that the integral $\int_{\r0}\phi(z)L(.,dz)$ is quasi-continuous on $\R_+$ and that the support of $\phi$ lies in an open set $A$, $0\notin \bar A$ with $c^{\v}(\partial A)=0$. Then of course the stopping times $\tau^k_A$ are q.c. and it wouldn't be difficult to prove by stopping the integral that $\phi(\Delta X_{\tau^k_A})$ would also be q.c. By Theorem \ref{tw_qc_of_phi_d} and Remark \ref{rem_qc_phi_d} we would get the $\v$-q.c. of $\phi$. However, the assumption of q.c. of the integral is too strong, as it is show in the next proposition.
\end{rem}
\begin{prop}
	Let $\phi\in C_{b}(\r0)$ with a support bounded away from $0$ and of positive $\v$-capacity. Then the process $t\mapsto \int_{\r0}\phi(z)L(t,dz)$ is not quasi-continuous neither on $\R_{+}$ nor on any $[0,T],\ T>0$.
\end{prop}
\begin{proof}
	The proof will be given only for $I=\R_{+}$, as the other case follows exactly the same argument. We assume the contrary. First note that by Proposition 19 in \cite{Denis_function_spaces} and Corollary \ref{cor_int_qc_L1c_d} we have that for each $\epsilon>0$ there exists $\delta>0$ such that if $A\in \B(\Omega)$ with $c(A)\leq \delta$ then
	\begin{equation}\label{eq_uniform_integrab_d}
			\GE[\int_{\r0}|\phi(z)|L(t,dz)\I_{A}]\leq \epsilon
	\end{equation}

	 Hence we fix $t$, $\epsilon>0$ and take $\delta$ as above. By assumed quasi-continuity we can choose an open set $O$ such that $c(O)<\delta$ and $(t,\omega)\mapsto \int_{\r0}\phi(z)L(t,dz)(\omega)$ is continuous on $[0,\infty[\times O^c$. Let $A$ be the support of $\phi$. Fix any $r>0$ and $\omega\in O^c$ and take a sequence $r_n\uparrow r$. By the assumption we have that
	\[
		\sum_{0<u\leq r_n} \phi(\Delta \omega_u)\to \sum_{0<u\leq r} \phi(\Delta \omega_u),
	\]
	hence $\phi(\Delta \omega_r)=0$. But $r$ and $\omega$ were arbitrary, hence
	\[
		\phi(\Delta X.(.))\equiv 0
	\]
	on $[0,\infty[\times O^c$ and we conclude that $\int_{\r0}|\phi(z)|L(.,dz)\equiv 0 $ on $[0,\infty[\times O^c$.
	
	Now we know that $\GE[\int_{\r0}|\phi(z)|L(t,dz)]=t\sup_{v\in\v}\int_{\r0}|\phi(z)|v(dz)$. Moreover, by assumption that the support of $\phi$ has positive $\v$-capacity we know that this expectation must be also positive.
	
	However, by \eqref{eq_uniform_integrab_d} we have that
	\[
		\GE[\int_{\r0}|\phi(z)|L(t,dz)]\leq \GE[\int_{\r0}|\phi(z)|L(t,dz)\I_{O}]+\GE[\int_{\r0}|\phi(z)|L(t,dz)\I_{O^c}]\leq \epsilon.
	\]
	We get the contradiction hence the integral cannot be quasi-continuous on $\R_+$.
\end{proof}

It is trivial to see why there is a problem with non-quasi-continuity of a Poisson integral as a stochastic process: for every path it has "large" discontuitities at jump times. However smoothing it a bit by making "new" jumps "small", helps to obtain quasi-continuity without distorting the process to much, hence the stopping time technique might be still successfully applied to get the desired result.
\begin{tw}\label{tw_integral_qc_on_R_d}
	Fix a bounded function $\phi $ such that the support of $\phi$ lies in an open set $A$ with $c^{\v}(\partial A)=0$, $c^{\v}(A)>0$ and $0\notin\bar A$. Assume for each $t>0$ we can find a sequence of functions $f^t_n\colon \R_+\to \R_+$ such that
	\begin{enumerate}
		\item $f^t_n$ are continuous, non-increasing and $f^t_n\downarrow \I_{[0,t]}$,
		\item $f^t_n(u)=0$ for $u\geq t$, $f^t_n(u\vee 0)=1$ for $u\leq t-1/n$.
		\item  the stochastic process $Y^n$ defined as
		\[
			Y^n_t:=\int_0^t\int_{\r0}\phi(z)f^t_n(s)L(ds,dz):=\sum_{0<u\leq t}\phi(\Delta X_u)f^t_n(u)
		\]
		is quasi-continuous on $\R_{+}$ for all $n$.
	\end{enumerate}
	Then $\phi(\Delta X_{\tau^k_A})$ is q.c. for all $k$. Consequently,  $\phi$ is $\v$-q.c.
\end{tw}

Before we prove the theorem we note that it is easy to check that for any $\phi\in C_{b}(\r0)$ with the support bounded away from $0$ we can easily find functions $f^t_n$ such that $Y^n$ is continuous on $[0,\infty[\times \Omega$. If $\phi$ is $\v$-q.c. and bounded with support bounded away from $0$ one can also check that the assumptions above are satisfied.
\begin{prop}
	Let $\phi$ be a bounded, $\v$-q.c. function such that the support of $\phi$ lies in an open set $A$ with $c^{\v}(\partial A)=0$, $c^{\v}(A)>0$ and $0\notin\bar A$. Then the process $Y^n_t$ defined in Theorem \ref{tw_integral_qc_on_R_d} is quasi-continuous on $\R_+$ for any family of functions $f_n^t$ satisfying properties 1-2. 
\end{prop}
\begin{proof}
	First we fix $T>0$ and we prove the quasi-continuity on $[0,T]$. Fix $\epsilon>0$ and an open subset $\tilde O_T\subset \r0$ such that $c^{\v}(\tilde{O}_T)<\epsilon/T$ and $\phi$ is continuous on $\tilde{O}_T^c$. Define $O_T$ as follows
	\[
		O_T:=\{\omega\in\Omega\colon \exists\, u\leq T\ \Delta \omega_u\in \tilde O_T\cap  A\}.
	\]
	Hence $O^c_T=\{\omega\in\Omega\colon \forall\, u\leq T\ \Delta \omega_u\in \tilde O^c_T\cup  A^c \}$ and it is a closed set. Moreover it is easy to see that $Y^n_t$ is continuous on $[0,T]\times O^c_T$ if $f_n^t$ satisfies the properties 1 and 2. Note also that $c(O_T)<\epsilon$ by exactly the same argument as in the proof of Theorem \ref{tw_qc_of_phi_d}, Point 2.
	Therefore, $Y^n$ is q.c. on $[0,T]$. But $T$ was arbitrary. So fix $\epsilon>0$ again and for each $T=n,\ n=1,2\ldots$ we can choose an open set $O_n$ with capasity $c(O_n)\leq \epsilon/2^n$ such that $Y^n$ is continuous on $[0,T]\times O_n^c$. By putting $O=\bigcup_{n=1}^{\infty}$ we immediately see that $Y^n$ is also continuous on $\R_{+}\times O^c$. Therefore $Y^n$ is q.c. on $\R_+$.
\end{proof}

\begin{proof}[Proof of Theorem \ref{tw_integral_qc_on_R_d}]
	We follow the idea presented in Remark  \ref{rem_idea_of_proof_qc_d}. We fix $k,n\in\N$ and $\epsilon>0$. By the quasi-continuity of $Y^n$ we may find an open set $O_1$ with $c(O_1)<\epsilon/2$ such that $Y^n$ is continuous on $[0,\infty[\times O^c_1$. By the properties of the set $A$ and Propositions \ref{prop_tau_qc_d} and \ref{prop_qc_different_stop_times_d} there exists an open set $O_2$ with capacity $c(O_2)<\epsilon/2$ such that $\tau^k_A$, $\tau^{k+1}_A$, $\I_{\{\tau_A^{k-1}\leq \tau_A^k-1/n\}}$ and $\I_{\{\tau_A^{k}\leq \tau_A^{k+1}-1/n\}}$ are continuous on $O^c_2$.
	
	Take $O=O_1\cup O_2$. Then we have that by the choice of $O$ and the definition of $Y^n$ that
	\begin{align*}
		Z^n:=&\left(Y^n_{\tau^{k+1}_A}-Y^n_{\tau^{k}_A}\right)\I_{\{\tau_A^{k}\leq \tau^{k+1}_A-1/n\}\cap \{\tau_A^{k-1}\leq \tau^k_A-1/n\}}\\
		&=\phi(\Delta X_{\tau^k_A})\I_{\{\tau_A^{k}\leq \tau^{k+1}_A-1/n\}\cap \{\tau_A^{k-1}\leq \tau^k_A-1/n\}}
	\end{align*}
	is continuous on $O$. Consequently, by boundedness of $\phi$ we get that $Z^n\in L^1_G(\Omega)$. We will show now that $Z^n$ converges to $\phi(\Delta X_{\tau^k_A})$ in $L^1_G(\Omega)$, hence $\phi(\Delta X_{\tau^k_A})\in L^1_G(\Omega)$ and, in particular, is quasi-continuous.
	
	\begin{align*}
		\GE[|Z_n-\phi(\Delta X_{\tau^k_A})|]&=\GE[\phi(|\Delta X_{\tau^k_A})||1-\I_{\{\tau_A^{k}\leq \tau^{k+1}_A-1/n\}\cap \{\tau_A^{k-1}\leq \tau^k_A-1/n\}}|]\\
		&\leq B\, c(\tau_A^{k}> \tau^{k+1}_A-1/n)+B\, c(\tau_A^{k-1}> \tau^k_A-1/n),
	\end{align*}
	where $B$ is a bound of $\phi$. We will deal with the first summand, as the second has exactly the same structure. 
	
	Take again the set $O_2$. Note that 
	\[
		c(\tau_A^{k}> \tau^{k+1}_A-1/n)\leq c(\{\tau_A^{k}-\tau^{k+1}_A\geq -1/n\}\cap O_2^c)+c(O_2)<c(\{\tau_A^{k}-\tau^{k+1}_A\geq -1/n\}\cap O_2^c)+\frac{\epsilon}{2}
	\]
	Both $\tau_A^{k}$ and $\tau^{k+1}_A$ are continuous on $O_2^c$ so $F_n:=(\{\tau_A^{k}-\tau^{k+1}_A\geq -1/n\}\cap O_2^c$ is a sequence of closed sets decreasing to $\emptyset$ (as $\tau^k_A<\tau^{k+1}_A$ by its definition). Hence by the relative compactness of $\mathfrak{P}$ (see \cite{Ren}) and Theorem 12 in \cite{Denis_function_spaces} we may find $N>0$ s.t. for all $n>N$ $c(F_n)<\epsilon/2$. Consequently $Z_n\to \phi(\Delta X_{\tau^k_A})$ in $L^1_G(\Omega)$.
\end{proof}

\subsection{Characterization of the $\v$-quasi continuity}\label{ssec_v_qc_d}
As it was seen in the Subsection \ref{ssec_regular_stop_times_d}, the $\v$-quasi continuity (under some mild conditions) determine the quasi-continuity of the integral. In this subsection we will give the quick criterion for $\v$-quasi-continuity. 

\begin{prop}\label{prop_criterion_phi_v_qc_d}
	Let $\phi\in L^0(\r0)$ and let $A$ be the set of all discontinuity points of $\phi$. Then if $\v$ is relatively compact and $c^{\v}(\bar A)=0$, then $\phi$ is $\v$-q.c. On the other hand, if $\phi$ is $\v$-q.c. then \[\inf\{c^{\v}(O)\colon O\textrm{ - an open subset of }\r0,\ A\subset \bar O\}=0.\]
\end{prop}
\begin{proof}
	Assume that $c^{\v}(\bar A)=0$. Then by Lemma 3.4 in \cite{Song_hitting} we have that
	\[
		0=\inf\{c^{\v}(O)\colon \bar A\subset O, O\textrm{ is open}\},
	\]
	so for each $\epsilon>0$ we can choose an open set containing $\bar A$ with $\v$-capacity less than $\epsilon$. Of course $\phi$ is continuous on $O^c$, hence it's $\v$-q.c. Note that Lemma 3.4. is formulated for the case of a capacity normed by $1$. However the proof does not depend on this property and only the relative compactness is crucial for it. 
	
	To prove the second assertion of the proposition, take $\epsilon>0$. By the $\v$-quasi-continuity there exists an open set $O$ such that $c^{\v}(O)<\epsilon$ and $\phi$ is continuous on $O^c$.
	
	Take any $x\in A$. By the definition of $A$ there exists a sequence $x_n\to x$ s.t. $\phi(x_n)$ doesn't converge to $\phi(x)$. We have two cases. 
	\begin{enumerate}
		\item Infinitely many $x_n$ belong to $O^c$. W.l.o.g. we can assume that all $x_n\in O^c$. By the closedness of $O^c$ we deduce that also $x\in O^c$. But this contradicts the continuity of $\phi$ on $O^c$.
		\item Infinitely many $x_n$ belong to $O$. Again w.l.o.g. we can assume that all $x_n\in O$. Hence $x\in \bar O$. But $x\in A$ was an arbitrary point, hence $A\subset\bar O$ and we have
		\[
			\inf\{c^{\v}(O)\colon \bar A\subset O, O\textrm{ is open}\}<\epsilon.\qedhere
		\]
	\end{enumerate}
\end{proof}
\subsection{Example of a Poisson integral which is not quasi-continuous}

In this subsection we consider a $G$-\levy process $X$ associated with the following set $\u:=\{\delta_x\colon x\in[1,2]\}\times\{0\}\times\{0\}$. We show directly from the definition that a very simple integral $Y_T:=\int_{\r0}\I_{\{1\}}(z)N(T,dz)$ is not quasi-continuous. Note that it is easy to check using Proposition \ref{prop_criterion_phi_v_qc_d} that $\phi=\I_{\{1\}}$ is not $\v$-q.c., just as it is predicted by the theory in Subsection \ref{ssec_regular_stop_times_d}.

\begin{prop}
	The random variable $Y_T$ is not q.c. for any $T>0$.
\end{prop}
\begin{proof}[Sketch of the proof]
		We fix $T>0$. For simplicity we take $\Omega=\mathbb D_0 ([0,T],\R^d)$ and we will drop the subscript $T$ in $Y_T$. Assume the contrary, i.e. that $Y$ is q.c. Fix $0<t_1<t_2<T$ and take $0<\epsilon<\Q(M([t_1,t_2],1)=1)$, where $M$ is a Poisson random measure associated with a $\Q$-Poisson process with parameter $1$. By the assumed quasi-continuity of  $Y$ we can choose a an open set $O$ with $c(O)<\epsilon$ s.t. $Y|_{O^c}$ is continuous.
		
	Introduce the following family of closed sets for $x\in[1,2]$
	\[
		\Omega^x:=\{\omega\in\Omega\colon \Delta\omega_s\in\{0,x\}\ \forall\, s\in[0,T]\}.
	\]
	Of course we have that $c(\Omega^x)=1$ for all $x\in[1,2]$. We can have two cases:
		\begin{enumerate}
			\item $O^c\cap \Omega^1=\emptyset$. Then $O\supset\Omega^1$ and consequently $1>\epsilon>c(O)\geq 1$. Contradiction.
			\item $O^c\cap \Omega^1\neq \emptyset$. Then w.l.o.g. we may assume that $O^c\supset \Omega^1$ as $Y$ is continuous on $\Omega^1$, which is a closed set. 
			For $t\in]0,T[$ and $x\in[1,2]$ fine the following paths:
			\[
				\omega^{t,x}:=0\cdot\I_{[0,t]}+x\cdot\I_{]t,T]}.
			\]
			Note that if $x_n\downarrow 1$ and $t_n\to t$ then $\omega^{t_n,x_n}\to \omega^{t,x}$, but $Y(\omega^{t_n,x_n})=0$ for all $n$ and $Y(\omega^{t,x})=1$. Hence for each $t\in]0,T[$ there exist constants $0\leq s^t<t<S^t\leq T$ and $1<A^t\leq 2$ s.t. for all $u\in]s^t,S^t[$ and $x\in]1,A^t]$ we have that $\omega^{u,x}\notin O^c$ (as $Y$ is continuous on $O^c$). Define $I^t:=]s^t,S^t[$. Then it is obvious that the family $\{I^t\colon t\in]0,T[\}$ constitutes an open covering of a comapct interval $[t^1,t^2]$ which was introduced earlier. Hence, we can choose the finite subcovering $\{I^t\colon t\in\{u_1,\ldots,u_n\}\}$. We also can define then $A=\min_{0\leq i\leq n}\, A^{u_i}$. Note that $A>1$. Consequently we have
			\[
				\{\omega^{u,x}\colon u\in[t^1,t^2],\ x\in]1,A]\}\subset O.
			\]
			But we also have the following
			\begin{align*}
				\epsilon>c(O)&\geq c(\{\omega^{u,x}\colon u\in[t^1,t^2],\ x\in]1,A]\})\geq \Q(M([t_1,t_2],1)=1)>\epsilon.
			\end{align*}
			The third inequality is the consequence of the fact that for each $x\in]1,A]$ there is a $\P^x\in\mathfrak{P}$ such that the canonical process under $\P^x$ is a standard Poisson process multiplied by $x$ and $\P^x(\{\omega^{u,x}\colon u\in[t^1,t^2],\ x\in]1,A]\})=\Q(M([t_1,t_2],1)=1)$.
		\end{enumerate}
		In both cases we obtained the contradiction, hence $Y$ cannot be q.c.
\end{proof}

\section{Applications}\label{sec_applications_d}

In this section we will apply the results of the previous sections to the problem of the decomposition of $G$-\levy processes.  We will prove that we can require the jump part to be a $G$-martingale in $L^1_G(\Omega)$, but we cannot make the jump part a symmetric $G$-martingale (unless we extend the canonical space). 

First, we have the following easy corollary of Corollary \ref{cor_int_qc_L1c_d} and Proposition \ref{prop_criterion_phi_v_qc_d}.
\begin{cor}\label{cor_integral_qc_with_A_d}
	Let $A\subset \r0$ be such that $0\notin \bar A$ and $c^{\v}(\partial A)=0$. Let $\phi$ be a continuous function with linear growth. Then under Assumption \ref{ass_uniform_integ_d} we have that $\int_{A}\phi(z)L(t,dz)$ is in $L^1_G(\Omega_t)$ for all $t>0$. In particular, the integral $\int_A z L(t,dz)\in L^1_G(\Omega_t)$.
\end{cor}
\begin{proof}
	It is sufficient to prove that $\I_A\phi\in \L^1_c(A,\v)$. We use the characterization of functions in $\L^1_c(A,\v)$ (compare with Remark \ref{rem_cv_capacity_d} and Proposition \ref{prop_characterization_spaces_v_d}). By the assumption on $A$ we know that there exists $\epsilon>0$ such that $\bar A\subset B(0,\epsilon)^c$. We also know that $\v$ restricted to $B(0,\epsilon)^c$ is relatively compact (compare with \cite{Peng_levy}, p. 14), hence the $c^{\v|_{B(0,\epsilon)^c}}(\partial A)=0$ and we can prove the $\v|_{B(0,\epsilon)^c}$-q.c. of $\I_A\phi$, which then we easily extend to $\v$-quasi-continuity. The "uniform integrability condition" might be easily obtained via the the linear growth and the fact that $\sup_{v\in\v}\int_{\{|z|\geq1\}}|z|^pv(dz)<\infty$ for some $p>1$. 
\end{proof}
Very similarly we get the following theorem.
\begin{tw}\label{tw_decomposition_in_L1G_d}
	Let $X$ be a $G$-\levy process defined on the canonical sublinear expectation space $(\Omega,L^1_G(\Omega),\GE[.]),\ \Omega=\d0$, satisfying Assumption \ref{ass_uniform_integ_d}. 
	\begin{enumerate}
	\item Then the decomposition $X_t=X^c_t+X^d_t$ as in Point 4 of Definition \ref{def_levy_d} might be taken on the same sublinear expectation space with $X_t^c:=X_t-\int_{\r0}zL(t,dz)$ and $X^d_t:=\int_{\r0}zL(t,dz)$ and both processes  belong to $L^1_G(\Omega)$ for each $t$. 
	\item We may also require the discontinuous part to be a $G$-martingale without losing the property that it belongs to $L^1_G(\Omega)$. We simply take $X^d_t:=\int_{\r0}zL(t,dz)-t\sup_{v\in\v}\int_{\r0}zv(dz)$. However, such defined discontinuous part is not a symmetric $G$-martingale, unless $\v=\{v\}$. 
	\item If there exist disjoint sets $A_1,\ldots,A_n\subset \r0$ such that $0\notin \bar A_i$  and $c^{\v}(\partial A_i)=0$ for each $i=1,\ldots,n$ then the following process \[t\mapsto(X_t-\sum_{i=1}^n\int_{A_i}zL(t,dz),\int_{A_1}zL(t,dz),\ldots,\int_{A_n}zL(t,dz))\] is a $\tilde G^{A_1,\ldots,A_n}$-\levy process on sublinear expectation space $(\Omega,L^1_G(\Omega),\GE[.])$, with $\Omega=\d0$, where $\tilde G^{{A_1,\ldots,A_n}}$ is a non-local operator associated with set $\tilde\u$ via \levykhintchine formula in Theorem \ref{tw_levykhintchine_d} where
	\[
	 \tilde\u:=\{(v|_{\bigcup_{i=1}^n A_i^c}\otimes v|_{A_1}\otimes\ldots\otimes v|_{A_n}, (p,0,\ldots,0),(q,0,\ldots,0))\colon (v,p,q)\in\u\}.
	\]
	\item In a similar manner we may compensate each of the discontinuous components of the $\tilde G^{A_1,\ldots,A_n}$-\levy process defined in Point 3.
	\end{enumerate}
\end{tw}
Lastly, we show that we cannot compensate the discontinuous part of $X$ with a factor which would make it a symmetric $G$-martingale without extending the space.
\begin{tw}
	Let $X$ be a $G$-\levy process with finite activity defined on the canonical sublinear expectation space $(\Omega,L^1_G(\Omega),\GE[.]),\ \Omega=\d0$. Assume that set $\u$ is of the following form
\[
	\u=\v\times\{0\}\times\q.
\]
and that there exists a measure $\pi$ on $\B(\r0)$ such that each $v\in\v$ is equivalent to $\pi$ and we have the following bounds $0<\underline{c}\leq\bar c<\infty$ for all $B\in\B(\r0)$
	\[
		\underline{c} \pi(B)\leq v(B)\leq \bar c\pi(B).
	\]
Assume also that there exists $p>2$ s.t.
\begin{equation}\label{eq_moment_ass_p_d}
	\sup_{v\in\v}\int_{\r0}|z|^pv(dz)<\infty.
\end{equation}

 Define $X^d_t:=\int_{\r0}zL(t,dz)$. If there exists a process with finite variation $Y$ such that $Y_t\in L^q_G(\Omega_t)$ some $q>2$ and $\tilde X^d_t:= X^d_t-Y_t$ is a symmetric $G$-martingale then $\v=\{v\}$.
\end{tw}
\begin{proof}
	Assume that there exists such a process.  Fix $T>0$. Note that by the Kunita's inequality (see Corollary 4.4.24 in \cite{applebaum}) and the moment assumption in eq. \eqref{eq_moment_ass_p_d} we have that $ X^d_t\in L^s_G(\Omega_t)$ for $2<s<p$. Then $\tilde X^d_T \in L^r_G(\Omega_T)$ for some $r=\min\{s,q\}>2$. Note that and by Theorem 25 and Proposition 26 in \cite{moj2} we get that $\tilde X^d_t$ must be written as a sum of a stochastic integral w.r.t. a $G$-Brownian motion, a non-increasing continuous $G$-martingale $K^c$ and a \itolevy integral compensated by its mean. Since $\tilde X^d_t$ has a finite variation we deduce that the first summand is equal to $0$, hence $\tilde X^d_t=\int_0^t\int_{\r0}K(s,z)L(ds,dz)-\int_0^t\sup_{v\in\v}\int_{\r0}K(s,z)v(dz)ds+ K^c_t$. But the latter is a symmetric $G$-martingale iff $\v=\{v\}$ and $K^c\equiv 0$.
\end{proof}

\section{Appendix}
\subsection{Constraction of $g_v$}\label{ssec_g_v_d}
Take any measure $\mu\in\mathcal M(\R^{d})$ s.t. $\mu$ is a \levy measure absolutely continuous w.r.t. Lebesgue measure and $\mu(\r0)=\infty$. 

Fix a sequence $\{\epsilon_n\}_{n\geq 0}$ s.t. $\epsilon_0=\infty$ and $\epsilon\downarrow0$ as $n\to \infty$. 

Let $r_n:=\sup_{v\in\v}v(\{\epsilon_n<|z|\leq \epsilon_{n-1}\}),\ n=1,2,\ldots$. We know that $r_n<\infty$ (see p. 14 in \cite{Peng_levy}). By the assumptions on $\mu$ we may find a decreasing sequence $\{\eta_n\}_{n\geq 0}$ s.t. $\eta_0=\infty$ and $\mu(\{\eta_n<|z|\leq \eta_{n-1}\})=r_n,\ n=1,2,\ldots$. We introduce the notation
\[
	O_n:=\{z\in\r0\colon \epsilon_n<|z|\leq \epsilon_{n-1}\}\quad \textrm{and } 
	U_n:=\{z\in\r0\colon \eta_n<|z|\leq \eta_{n-1}\},\quad n=1,2,\ldots.
\]
Note that $v(O_n)\leq \mu(U_n),\ n=1,2,\ldots$ for every $v\in\v$.  Again by the properties of $\mu$ we may find a subset $U_n^v\subset U_n$ such that $v(O_n)=\mu(U_n^v)=:r_n^v$ for all $v\in\v$ and $n=1,2,\ldots$.

Since $v|_{O_n}/r^v_n$ and $\mu|_{U^v_n}/r^v_n$ are two probability measures and the second one is absolutely continuous w.r.t. Lebesgue measure, we may use the Knothe-Rosenblatt rearrangement (see for example \cite{Villiani}, p.8-9) to find a function $g_{v,n}\colon U_n^v\to O_n$ such that
\[
	\frac{v|_{O_n}(A)}{r^v_n}=\frac{\mu|_{U^v_n}(g^{-1}_{v,n}(A))}{r^v_n}\quad \forall\ A\in \B(O_n).
\]

It is now trivial that by putting $g_v:=g_{v,n}$ on every $U^v_n$ and $g_v\equiv 0$ outside $\bigcup_{n} U^v_n$ we get a function which transports measure $\mu$ onto $v$. Of course by the construction for each $\epsilon>0$ there exists an $\eta>0$ s.t.
\[
	\bigcup_{v\in\v}g^{-1}_v(B(0,\epsilon)^c)\subset B(0,\eta)^c.
\]
We may also take $\mu$ which integrates $|z|$.

\subsection{Proof of characterization of $L^1_G(\Omega)$}
Before we will go with the proof let us remind the properties of \cadlag modulus
\begin{lem}\label{lem_modulus_d}
	For any $\delta>0$ and a \cadlag function $x:[0,T]\to \R^d$ define the following \cadlag modulus
	\[
		\omega_x'(\delta):=\inf_{\pi}\max_{0<i\leq r} \sup_{s,t\in [t_{i-1},t_i[} |x(s)-x(t)|,
	\]
	where infinimum runs over all partitions $\pi=\{t_0,\ldots,t_r\}$ of the interval $[0,T]$ satisfying $0=t_0<t_1<\ldots<t_r=T$ and $t_i-t_{i-1}>\delta$ for all $i=1,2,\ldots,r$.
	Define also
		\begin{align*}
		w_x''(\delta):=&\sup_{\substack{t_1\leq t\leq t_2\\t_2-t_1\leq \delta}}\min\{ |x(s)-x(t_1)|,|x(t_2)-x(s)|\}.
	\end{align*}
	Then
	\begin{enumerate}
		\item 		$w_x''(\delta)\leq w_x'(\delta)$ for all $\delta>0$ and $x\in \d0$.
		\item For every $\epsilon>0$ and a subinterval $[\alpha,\beta[\subset [0,T]$ if  $x$ does not have any jumps of magnitude $>\epsilon$ in the interval $[\alpha,\beta[$ then
		\[
			\sup_{\substack{t_1,t_2\in[\alpha,\beta[\\|t_2-t_1|\leq\delta}} |x(t_1)-x(t_2)|\leq 2w''_x(\delta)+\epsilon\leq 2w'_x(\delta)+\epsilon.
		\]
		\item The function $x\mapsto w_x'(\delta)$ is upper semicontinuous for all $\delta>0$.
		\item  $\lim_{\delta\downarrow 0}\,w_x'(\delta)\downarrow 0$ for all $x\in \d0$.
	\end{enumerate}
\end{lem}
These properties are standard and might be found in \cite{billingsley} for properties 1, 3 and 4 (see Chapter 3, Lemma 1, eq. (14.39) and (14.46)) and \cite{Parthasarathy} for property 2 (see Lemma 6.4 in Chapter VII).
\begin{proof}[Proof of Proposition \ref{prop_lip_in_Lg_d}]
	Fix a random variable $Y\in C_{b,lip}(\Omega_T)$. For any $n\in\N$ define the operator $T^n\colon\d0\to \d0$ as
	\[
		T^n(\omega)(t):=\left\{ \begin{array}{ll}
\omega_{\frac{kT}{n}} & \textrm{if }t\in[\frac{kT}{n},\frac{(k+1)T}{n}[,\ k=0,1,\ldots,n-1.\\
\omega_T & \textrm{if }t=T.
\end{array} \right.
	\]
	Define $Y^n:=Y\circ T^n$. Then $Y^n$ depend only on $\{\omega_{{kT}/{n}}\}_{k=0}^{n}$ thus there exists a function $\phi^n\colon \R^{(n+1)\times d}\to\R$ such that
	\[
		Y^n(\omega)=\phi^n(\omega_0,\omega_{\frac{T}{n}},\ldots,\omega_{T}).
	\]
	By the boundedness and Lipschitz continuity of $Y$ we can easily prove that also $\phi^n$ must be bounded and Lipschitz continuous (all we have to do is to consider the paths, which are constant on the intervals $[kT/n,(k+1)T/n[$). Note however that
	\begin{align*}
		\GE[|Y-Y^n|]=\GE[|Y-Y\circ T^n|]\leq L\,\GE[d(X^T, X^T\circ T^n)\wedge 2K],
	\end{align*}
	where $L>0$ and $K\cdot L$ are respectively a Lipschitz constant and bound of $Y$, $X^T$ is a canonical process, i.e. our $G$-\levy process, stopped at time T and $d$ is the Skorohod metric.
	
Fix now $\epsilon>0$. Then for any $\omega\in \Omega_T$ we have that a number of jumps with the magnitude $>\epsilon$ is finite. Fix $\omega\in\Omega_T$ and let $0<r_1<\ldots<r_{m-1}< T$ be the times of jumps with magnitude $>\epsilon$. We can possibly have such a jump also at $r_m:=T$. We can choose $n$ big enough such that $r_{i+1}-r_i\geq {T}/{n}$ for $i=0,\ldots, m-1$. Define $A_T^{n,\epsilon}$ as a set of all $\omega\in \Omega_T$ for which the minimal distance between jumps of magnitude $>\epsilon$ is larger or equal to ${T}/{n}$. We want to have an estimate of the Skorohod metric for $\omega\in A^{n,\epsilon}_T$. To obtain it we construct the piecewise linear function $\lambda^n$ as follows $\lambda^n(0)=0$, $\lambda^n(T)=T$, for each $k=1,\ldots,n-1$ define
	\[
		\lambda^n\left(\frac{kT}{n}\right) :=\left\{ \begin{array}{ll}
 \frac{kT}{n} & \textrm{if } r_i\notin \left] \frac{(k-1)T}{n},\frac{kT}{n}\right],\ i=1,\ldots,m ,\\
r_i & r_i\in \left] \frac{(k-1)T}{n},\frac{kT}{n}\right],\ i=1,\ldots,m.
\end{array} \right.
	\]
	Moreover, let $\lambda^n$ be linear between these nods. By the construction $\|\lambda^n-Id\|_{\infty}\leq 2T/n$. Define $t_k:=\lambda^n({kT}/{n})$ for $k=0,\ldots,n$. Note that $\omega$ does not have any jump of magnitude $>\epsilon$ on $[t_k,t_{k+1}[$.
	Then by definition of the Skorohod metric and property 2 in Lemma \ref{lem_modulus_d} we have	
	\begin{align*}
		d(\omega,T^n(\omega))\wedge 2K&=\left(\inf_{\lambda\in\Lambda}\max\{\|\lambda-Id\|_{\infty},\|T^n(\omega)-\omega\circ\lambda\|_{\infty}\}\right)\wedge2K\\
		&\leq \left(\|\lambda^n-Id\|_{\infty}+\|T^n(\omega)-\omega\circ\lambda^n\|_{\infty}\}\right)\wedge2K\\
		&\leq \left(\frac{2T}{n}+\max_{k=0,\ldots,n-1} \sup_{s,t\in[t_k,t_{k+1}[} |\omega(s)-\omega(t)|\right)\wedge 2K\\
		&\leq \left[\frac{2T}{n}+2w'_{\omega}\left(\frac{2T}{n}\right)+\epsilon\right]\wedge 2K.
	\end{align*}
	Thus we can define yet another bound $K^{n,\epsilon}$ as
    \[
        K^{n,\epsilon}(\omega):=\left\{
                       \begin{array}{ll}
                         \left(\frac{2T}{n}+2w'_{\omega}\left(\frac{2T}{n}\right)+\epsilon\right)\wedge 2K, & \textrm{if }\omega\in A_T^{n,\epsilon},\\
                         2K, & \textrm{if }\omega\in \Omega_T\setminus A_T^{n,\epsilon},\\
                       \end{array}
                     \right.
    \]
    Then $d(X^T, X^T\circ T^n)\wedge 2K\leq K^{n,\epsilon} $ and thus $\GE[d(X^T, X^T\circ T^n)\wedge 2K]\leq \GE[K^{n,\epsilon}]$. We also have $K^{n,\epsilon}\downarrow \epsilon$ as $n\to\infty$ on every $A^{m,\epsilon}_T$, $m$ and $\epsilon$ are fixed. This follows from property 4 in Lemma \ref{lem_modulus_d}. Moreover we claim than $K^{n,\epsilon}$ is upper semi-continuous on every set $A_T^{m,\epsilon}$ for $m\leq n$ and a fixed $\epsilon$. Firstly, note that the set $A^{m,\epsilon}_T$ is closed under the Skorohod topology. This is clear from the definition of the set: if $\{\omega^k\}_k\subset A^{m,\epsilon}_T$ then the distance between the jumps of magnitude $>\epsilon$ is $\geq {T}/{m}$ for each $k$. But if $\omega^k\to\omega$ then also $\omega$ must satisfy this property\footnote{Note that it is not a problem for us that a jump of magnitude $>\epsilon$ for $\omega^k$ may have jump of magniture exactly $\epsilon$ in the limit, as then the distance between "large" jumps would only increase.} and hence it belong to $A^{m,\epsilon}_T\subset A^{n,\epsilon}_T$. Now note that by Lemma \ref{lem_modulus_d}, property 3, we have that $\omega\mapsto\left({2T}/{n}+2w'_{\omega}\left({2T}/{n}+\epsilon\right)\right)\wedge 2K$ is upper semi-continuous as a minimum of two upper semi-continuous functions and thus
        \begin{align*}
            \limsup_{k\to\infty} K^{n,\epsilon}(\omega^k)&= \limsup_{k\to\infty}  \left(\frac{2T}{n}+2w'_{\omega^k}\left(\frac{2T}{n}\right)+\epsilon\right)\wedge 2K\\&\leq \left(\frac{2T}{n}+2w'_{\omega}\left(\frac{2T}{n}\right)+\epsilon\right)\wedge 2K= K^{n,\epsilon}(\omega).
        \end{align*}
Thus $K^{n,\epsilon}$ is upper semi-continuous on each closed set $A^{m,\epsilon}_T,\ m\leq n$. Then we have that

We also claim that the sets $A_T^{m,\epsilon}$ are 'big' in the sense, that the capacity of the complement is decreasing to $0$. Note that
\[
	(A_T^{m,\epsilon})^c=\{\omega\in\Omega_T\colon  \exists\, t,s\leq T,\ |t-s|<\frac{T}{m} \textrm{ and }|\Delta \omega_t|>\epsilon,\ |\Delta \omega_s|>\epsilon\}.
\]
For any $\theta \in \a^{\u}_{0,T}$ define the set
\[
	(A_T^{m,\epsilon,\theta})^c=\{\omega\in\Omega_T\colon  \exists\, t,s\leq T,\ |t-s|<\frac{T}{m} \textrm{ and }|\Delta B^{0,\theta}_t(\omega)|>\epsilon,\ |\Delta B^{0,\theta}_s(\omega)|>\epsilon\}.
\]
We want to use the definition of $c$ and the fact that $\P^{\theta}$ is the law of $B.^{0,\theta}$. We remind that $N_t=\int_{\r0}zN(]0,t],dz)$. If $B.^{0,\theta}$ has a jump of magnitude greater than $\epsilon$ then $N$ must have also a jump at $t$ of magnitude greater then $\eta^{\epsilon}>0$ (compare with Remark \ref{rem_g_v_d} and Subsection \ref{ssec_g_v_d}). 
\begin{align*}
	c\left[(A^{m,\epsilon}_T)^c\right]&=\sup_{\theta\in\a^{\u}_{0,T}} \P^{\theta}\left[(A^m_T)^c\right]=\sup_{\theta\in\a^{\u}_{0,T}} \P_0\left[(A^{m,\theta}_T)^c\right]\\
	&\leq \P_0( \exists\, t,s\leq T,\ |t-s|<\frac{T}{m} \textrm{ and }|\Delta N_t|>\eta^{\epsilon} \textrm{ and } |\Delta N_s|>\eta^{\epsilon})=:\P(B^{m,\epsilon}).
\end{align*}
$N$ has ($\P_0-a.a.$) paths \cadlag  and hence  $B^{m,\epsilon}\supset B^{m+1,\epsilon}$ and $\P_0(\bigcap_{m=1}^{\infty}B^{m,\epsilon})=0$. Hence, by continuity of probability we have $\P_0(B^{m,\epsilon})\to 0$ and consequently for every $\epsilon>0$ one has 	$c\left[(A^{m,\epsilon}_T)^c\right]\to 0$ as $m\to\infty$.

Note that we will prove the assertion of our proposition if we use the following lemma (proof below).
\
\begin{lem}\label{lem_convergence_usc_d}
    For every $\epsilon>0$ let $\{X_{n,\epsilon}\}_n$ be a sequence of non-negative uniformly bounded random variables on $\Omega_T$ such that there exists a sequence of closed sets $(F_{m,\epsilon})_m$ having the following properties
    \begin{enumerate}
    \item $c(F_{m,\epsilon}^c)\to 0$ as $m\to\infty$.
    \item $X_{n,\epsilon}\downarrow \epsilon$ on every $F_{m,\epsilon}$.
    \item  $X_{n,\epsilon}$ is upper semi-continuous on every $F_{m,\epsilon}\ m\leq n$. 
\end{enumerate}     
Then for all $\epsilon>0$ there exists $N(\epsilon)$ such that for all $n>N(\epsilon)$ we have $\GE[X_{n,\epsilon}]<2\epsilon$.
\end{lem}

Applying this lemma to our sequence $\{K^{n,\epsilon}\}_n$ together with the closed sets $(A_T^{m,\epsilon})_m$ we get that for all $n$ big enough we have
\[
    \GE[|Y^n-Y|]\leq L\GE[ d(X^T,X^T\circ T^n)\wedge 2K]\leq L\GE[K^{n,\epsilon}]<2L\epsilon.\qedhere
\] 
\end{proof}

\begin{proof}[Proof of Lemma \ref{lem_convergence_usc_d}]
    Fix $\epsilon>0$. Let $M$ be the bound of all $X_{n,\epsilon}$. By the representation of the sublinear expectation we have
\begin{align*}
    \GE[X_{n,\epsilon}]&=\sup_{\theta\in\a^{\u}_{0,T}}\E^{\P^{\theta}}\left[X_{n,\epsilon}\right]=\sup_{\theta\in\a^{\u}_{0,T}}\int_0^M \P^{\theta}( X_{n,\epsilon}\geq t)d\leq\epsilon+\sup_{\theta\in\a^{\u}_{0,T}}\int_{\epsilon}^M \P^{\theta}( X_{n,\epsilon}\geq t)dt
     \\&=\epsilon+\sup_{\theta\in\a^{\u}_{0,T}}\int_{\epsilon}^M \P^{\theta}[ (\{X_{n,\epsilon}\geq t\}\cap F_{m,\epsilon})\cup (\{X_{n,\epsilon}\geq t\}\cap F_{m,\epsilon}^c)]dt\\
&\leq\epsilon+\sup_{\theta\in\a^{\u}_{0,T}}\int_{\epsilon}^M \P^{\theta}( \{ X_{n,\epsilon}|_{F_{m,\epsilon}}\geq t\}\cup F_{m,\epsilon}^c)dt
\leq\epsilon+\sup_{\theta\in\a^{\u}_{0,T}}\int_{\epsilon}^M \left[c(  X_{n,\epsilon}|_{F_{m,\epsilon}}\geq t)+c(F_{m,\epsilon}^c)\right]dt\\
&\leq\epsilon+\int_{\epsilon}^M c(  X_{n,\epsilon}|_{F_{m,\epsilon}}\geq t)dt+Mc(F_{m,\epsilon}^c).
\end{align*}
By the first property of sets $F_{m,\epsilon}$ we can choose $m$ big enough so that $c(F_{m,\epsilon}^c)\leq \frac{\epsilon}{2M}$.  Choose $n\geq m$. By the upper semi-continuity of $X_{n,\epsilon}$ on $F_{m,\epsilon}$ we get that each $\{X_{n,\epsilon}|_{F_{m,\epsilon}}\geq t\}$ is closed in the subspace topology on $F_{m,\epsilon}$. But $F_{m,\epsilon}$ is also a closed set in the Skorohod topology, thus $\{X_{n,\epsilon}|_{F_{m,\epsilon}}\geq t\}$ is also closed in it. Moreover, due to monotone convergence to $\epsilon$ on $F_{m,\epsilon}$ we have that $\{X_{n,\epsilon}|_{F_{m,\epsilon}}\geq t\}\downarrow \emptyset$ for every $t>\epsilon$ as $n\uparrow \infty$. Thus by Lemma 7 in \cite{Denis_function_spaces} we get that $c(X_{n,\epsilon}|_{F_m}\leq t)\downarrow0$ for every $t>\epsilon$ as $n\uparrow \infty$ and we get the assertion of the lemma by applying monotone convergence theorem for the Lebesgue integral and choosing $n\geq m$ big enough, so that the integral is less then $\frac{\epsilon}{2}$. Thus
\[
	0\leq \GE[X_{n,\epsilon}]\leq 2\epsilon\quad \textrm{for } n \textrm{ big enough}.\qedhere
\]
\end{proof}

\bibliographystyle{plain}
\bibliography{biblio_phd}
\end{document}